\documentclass[12pt,a4paper,reqno]{amsart}
\usepackage{color}
\usepackage{graphics,epic}
\usepackage{amsmath,amssymb, amsthm}
\usepackage[all,2cell]{xy}
\usepackage[numbers,sort&compress]{natbib}

\newtheorem{theorem}{Theorem}[section]
\newtheorem*{theorem*}{Theorem}

\newtheorem{lemma}[theorem]{Lemma}
\newtheorem{proposition}[theorem]{Proposition}
\newtheorem{corollary}[theorem]{Corollary}

\newtheorem*{conjecture*}{Conjecture}

\theoremstyle{remark}

\theoremstyle{definition}

\newtheorem{Notation}[theorem]{Notation}
\newcommand{\ca}{{\mathcal A}}

\newcommand{\cc}{{\mathcal C}}
\newcommand{\cd}{{\mathcal D}}
\newcommand{\ce}{{\mathcal E}}

\newcommand{\cg}{{\mathcal G}}

\newcommand{\ck}{{\mathcal K}}

\newcommand{\cm}{{\mathcal M}}

\newcommand{\ct}{{\mathcal T}}

\newcommand{\bt}{\boxtimes}
\newcommand{\ldot}{\dot{\Lambda}}
\newcommand{\lddot}{\ddot{\Lambda}}

\renewcommand{\hat}[1]{\widehat{#1}}

\newcommand{\al}{\alpha}
\newcommand{\ot}{\otimes}
\newcommand{\id}{{\rm id}}
\newcommand{\im}{{\rm im}}
\newcommand{\qdim}{{\rm qdim}\,}
\newcommand{\Hom}{{\rm Hom}\,}

\newcommand{\End}{{\rm End}\,}

\newcommand{\Ind}{{\rm Ind}\,}
\newcommand{\Aut}{{\rm Aut}\,}

\newcommand{\Z}{\mathbb{Z}}

\newcommand{\C}{\mathbb{C}}
\newcommand{\N}{\mathbb{N}}

\def\wt{{\rm wt}}

\newcommand{\la}{\lambda}

\def\C{{\mathbb C}}

\def\R{{\mathbb R}}

\def\Z{{\mathbb Z}}

\def\N{{\mathbb N}}

\def\1{{\bf 1}}
\def\la{{\langle}}
\def\ra{{\rangle}}
\def\tr{{\rm tr}}
\def \End{{\rm End}}
\def \Hom{{\rm Hom}}

\def \pf{\noindent {\bf Proof: \,}}
\def\theequation{5.\arabic{equation}}

\def \h{\mathfrak{h}}
\def \l{\lambda}
\def \w{\omega}
\def \g{\mathfrak{g}}

\begin{document}

\title[Holomorphic VOA of central charge 24 with Lie algebra $F_{4,6}A_{2,2}$]{A Holomorphic vertex operator algebra of central charge 24 with weight one Lie algebra $F_{4,6}A_{2,2}$}
\author{Ching Hung Lam}
\address{Ching Hung Lam, Institute of Mathematics, Academia Sinica, Taipei 10617, Taiwan}
\email{chlam@math.sinica.edu.tw}
\author{Xingjun Lin}
\address{Xingjun Lin, Collaborative innovation centre of Mathematics, School of Mathematics and Statistics, Wuhan University, Luojiashan, Wuhan, Hubei, 430072, China.}
\email{linxingjun88@126.com}

\thanks{C.\,H. Lam was partially supported by MoST grant 104-2115-M-001-004-MY3 of Taiwan}
\thanks{X. Lin was an ``Overseas researchers under Postdoctoral Fellowship of Japan Society for the Promotion of Science", and was supported by JSPS Grant No. 16F16020.}
\begin{abstract}
In this paper, a holomorphic vertex operator algebra $U$ of central charge 24 with the weight one Lie algebra $A_{8,3}A_{2,1}^2$ is proved to be unique. Moreover, a holomorphic vertex operator algebra of central charge 24 with weight one Lie algebra $F_{4,6}A_{2,2}$ is obtained by applying a $\Z_2$-orbifold construction to $U$. The uniqueness of such a vertex operator algebra is also established.
By a similar method, we also established the uniqueness of a holomorphic vertex operator algebra of central charge 24 with the weight one Lie algebra $E_{7,3}A_{5,1}$.
As a consequence, we verify that all $71$ Lie algebras in Schellekens' list can be realized as the weight one Lie algebras of some holomorphic vertex operator algebras of central charge $24$. In addition, we establish the uniqueness of three holomorphic vertex operator algebras of central charge $24$ whose weight one Lie algebras have the type
$A_{8,3}A_{2,1}^2$, $F_{4,6}A_{2,2}$, and  $E_{7,3}A_{5,1}$.
\end{abstract}
\maketitle
\section{Introduction}
\def\theequation{1.\arabic{equation}}
\setcounter{equation}{0}

The classification of strongly regular holomorphic vertex operator algebras (VOA) of central charge $24$ is an important problem in the theory of vertex operator algebras.
In 1993, Schellekens \cite{Sc93} obtained  a partial classification and determined the  possible Lie algebra structures for the weight one subspaces of holomorphic VOAs of central charge $24$ (see also \cite{EMS}). There are $71$ Lie algebras in his list but only $39$ of the $71$ cases in his list have been constructed explicitly at that time.
In the recent years, many new holomorphic VOAs of central charge $24$ have been constructed. In \cite{Lam,LS12}, a class of holomorphic VOAs called framed VOAs were studied. In particular, $17$ holomorphic VOAs were constructed. In addition,  holomorphic VOAs with weight one Lie algebras $E_{6,3}G_{2,1}^3$, $A_{2,3}^6$ and $A_{5,3}D_{4,3}A_{1,1}^3$ have been constructed in \cite{Mi3,SS}
using $\Z_3$-orbifold constructions associated with lattice VOAs.
Recently, van Ekeren, M\"oller and Scheithauer \cite{EMS} have established the general $\Z_n$-orbifold construction for elements of  arbitrary orders and the constructions of holomorphic VOAs with the weight one Lie algebras ${E_{6,4}}{C_{2,1}}{A_{2,1}}$, $A_{4,5}^2$,  $A_{2,6}D_{4,12}$, $A_{1,1}C_{5,3}G_{2,2}$ and $C_{4,10}$ were also discussed. In \cite{LS}, the constructions of  five other holomorphic VOAs have been obtained using an orbifold construction associated with inner automorphisms. Moreover, a holomorphic VOA of central charge $24$  whose weight one Lie algebra has the type $A_{6,7}$ has been constructed in \cite{LS16}.  Based on these results, $70$ of $71$ cases in Schellekens' list have been constructed. There is only one remaining case and the corresponding Lie algebra has the type $F_{4,6}A_{2,2}$.

In this article, we shall construct a strongly regular holomorphic vertex operator algebra of central charge $24$ whose weight one Lie algebra has the type $F_{4,6}A_{2,2}$. As a consequence, we verify that all $71$ Lie algebras in Schellekens' list can be realized as the weight one Lie algebras of some strongly regular holomorphic vertex operator algebras of central charge $24$. Our method is basically  a $\Z_2$-orbifold construction. We shall show that a strongly regular holomorphic VOA $\tilde{U}(g)$ of central charge $24$ with $\tilde{U}(g)_1=F_{4,6}A_{2,2}$ can be constructed by applying a $\Z_2$-orbifold construction to a holomorphic VOA $U$ with the weight one Lie algebra $A_{8,3}A_{2,1}^2$ and a suitable automorphism $g$ of order $2$. However, there are some fundamental differences between our method and the previous constructions for the other cases. In our construction, the fixed point of the automorphism $g$ on $U_1= A_{8,3}A_{2,1}^2$ should have the type $B_{4,6}A_{2,2}$. It means $g|_{U_1}$ is an outer automorphism of the Lie algebra $U_1$. In general, it is very difficult to determine if an outer automorphism  of Lie algebra $U_1$ can be lifted to an automorphism of the whole VOA. For $U_1=  A_{8,3}A_{2,1}^2$, there are at least two different constructions of a holomorphic VOA $U$ with $U_1= A_{8,3}A_{2,1}^2$. One construction is based on orbifold method  and is obtained in \cite{LS}. Another construction is based on mirror extensions of VOAs \cite{Xu,DJX1}. The construction based on mirror extensions is first obtained by Xu\,\cite{Xu} in terms of conformal nets and is proposed in \cite{DJX1} in VOA setting. In this article, we shall use the construction based on mirror extensions of VOAs. Using mirror extensions and the theory of modular invariants, we shall show that the VOA structure of a holomorphic VOA $U$ of central charge $24$ with $U_1= A_{8,3}A_{2,1}^2$  is unique, up to isomorphism (cf. Theorem \ref{main1}). We  also generalize a result of Shimakura \cite[Proposition 3.2]{S}, which gives a sufficient condition for lifting of an automorphism of a subVOA to the whole VOA (cf. Theorem \ref{lift1}). In addition, we determine the subgroup of $\Aut(U)$ which acts trivially on the weight one Lie algebra $U_1$ (see Section \ref{sec:4.4}).
By these facts, we are able to show if $U$ is a strongly regular holomorphic VOA  of central charge $24$ such that $U_1=A_{8,3}A_{2,1}^2$, then there exists an involution $g\in \Aut(U)$ such that $U^g_1$ is a Lie algebra of type $B_{4,6}A_{2,2}$. Moreover,   we are able to determine the conformal weights of the unique irreducible $g$-twisted $U^T(g)$ of $U$ using the explicit action of $g$ on $U_1$.
Finally, we shall establish the uniqueness for a holomorphic vertex operator algebra of central charge 24 with weight one Lie algebra $F_{4,6}A_{2,2}$  using the method of ``Reverse orbifold" proposed in \cite{LS16}. The key idea is to study the automorphisms of $\Aut(U)$  which act trivially on $U_1$ (see  Sections \ref{sec:4.4} and \ref{sec:7}). It turns out that the same method can be applied easily to the case when the weight one Lie algebra has the type $E_{7,3}A_{5,1}$ and we are able to establish the uniqueness for this case, also.

The organization of this article is as follows. In Section 2, we recall some basic facts about vertex operator algebras. In Section 3, we study some properties of the mirror extension $\widetilde{L_{sl_9}(3, 0)}$ of affine vertex operator algebra $L_{sl_9}(3, 0)$. The vertex operator algebra structure of $\widetilde{L_{sl_9}(3, 0)}$ is proved to be unique. We also show that the automorphism $\theta$ of $L_{sl_9}(3, 0)$ can be lifted to an automorphism of $\widetilde{L_{sl_9}(3, 0)}$. In Section 4, we study the structure of holomorphic vertex operator algebra $U$ of central charge 24 with Lie algebra $A_{8,3}A_{2,1}^2$. It is proved that holomorphic vertex operator algebra of central charge 24 with Lie algebra $A_{8,3}A_{2,1}^2$ is unique. As an application, we obtain an involution $\widetilde{\theta\ot \sigma}$ of $U$. In Section 5, we determine the conformal weights of $\widetilde{\theta\ot \sigma}$-twisted irreducible $U$-modules. In Section 6, we construct a holomorphic vertex operator algebra of central charge $24$ with Lie algebra $F_{4,6}A_{2,2}$ by orbifold construction. In Section 7, we prove that holomorphic VOAs  of central charge 24 with weight one Lie algebras $F_{4,6}A_{2,2}$ and $E_{7,3}A_{5,1}$ are unique, up to isomorphisms.

\section{Preliminaries}
\def\theequation{2.\arabic{equation}}
\setcounter{equation}{0}
\subsection{Basic definitions}
In this subsection, we shall recall some notations about vertex operator algebras from \cite{FHL,FLM,LL,Z}. Let $(V, Y(\cdot, z), \1, \w)$ be a vertex operator algebra as defined in \cite{FLM}. The vacuum vector and  the conformal element of $V$ are denoted by $\1$ and $\w$, respectively. The vertex operator $Y(v, z)$ corresponding to $v\in V$ is expanded as $Y(v,z)=\sum_{n\in \Z}v_nz^{-n-1}$. We also use the standard notation $L(n)$ to denote the component operator of $Y(\w,z)=\sum_{n\in \Z}L(n)z^{-n-2}$.
A linear automorphism $\sigma$ of $V$ is called an {\em automorphism} of $V$ if $\sigma(\1)=\1$, $\sigma(\w)=\w$ and $\sigma(u_nv)=\sigma(u)_n\sigma(v)$ for any $u, v\in V$, $n\in \Z$. We  denote the group of all automorphisms of $V$ by $Aut(V)$.

For a vertex operator algebra $V$, a {\em weak  $V$-module} is a vector space $M$ equipped
with a linear map
\begin{align*}
Y_{M}:V&\to (\End M)[[z, z^{-1}]],\\
v&\mapsto Y_{M}(v,z)=\sum_{n\in\Z}v_nz^{-n-1},\,v_n\in \End M
\end{align*}
satisfying a number of conditions (cf. \cite{DLM2}, \cite{FHL}). A weak
 $V$-module  $M$ is called an {\em admissible $V$-module} if $M$ has a $\Z_{\geq
0}$-gradation $M=\bigoplus_{n\in\Z_{\geq 0}}M(n)$ such
that
\begin{align*}
a_mM(n)\subset M(\wt{a}+n-m-1)
\end{align*}
for any homogeneous $a\in V$ and $m,\,n\in\Z$, where $\wt(a)=s$ if $a\in V_s$.
If any  admissible $V$-module is  a
direct sum of irreducible admissible modules, then $V$ is called {\em rational}.
It was proved in \cite{DLM2} that if $V$ is rational then there are only finitely many irreducible admissible $V$-modules up to isomorphism.

A {\em  $V$-module} is a weak $V$-module $M$ which carries a $\C$-grading induced by the spectrum of $L(0)$, that is,  $M=\bigoplus_{\lambda\in\C}
M_{\lambda}$ where
$M_\lambda=\{w\in M|L(0)w=\lambda w\}$. Moreover, one requires that $M_\lambda$ is
finite dimensional and for fixed $\lambda\in\C$, $M_{\lambda+n}=0$
for sufficiently small integer $n$.

A rational vertex operator algebra
is said to be {\em holomorphic} if it itself is the only irreducible module up to isomorphism. A vertex operator algebra $V$ is said to be of {\em CFT-type} if $V_0 = \C\1 $ (note that $V_n = 0$ for all $n < 0$ if $V_0 = \C\1$ \cite{DM}), and is said to be {\em $C_2$-cofinite} if the subspace $C_2(V)=\langle u_{-2}v|u, v\in V \rangle$ has finite codimension in $V$. A $V$-module $M = \bigoplus_{\lambda\in \mathbb{C}}{M_{\lambda}}$ is said to be {\em self-dual}
if $M$ is isomorphic to $M'$, where $M'$ denotes the $V$-module such that
$M'
= \bigoplus_{\lambda \in \mathbb{C}}{M_\lambda^*}$ and the vertex operator $Y_{M'}$ is defined
by the property
$$\langle Y_{M'}(a, z)u', v\rangle  = \langle u', Y_M(e^{zL(1)}(-z^{-2})^{L(0)}a, z^{-1})v\rangle ,$$for $a\in V, u'\in
M'$ and $v\in M$.  It is obvious that a holomorphic vertex operator algebra
is simple and self-dual. A vertex operator algebra is said to be {\em strongly regular} if it is simple, rational, $C_2$-cofinite and of CFT-type. Note that a strongly regular vertex operator algebra is self-dual \cite{Li1}.

Let $V$ be a CFT-type vertex operator algebra. It is well-known that $V_1$ has a Lie algebra structure such that $[u, v]=u_0v$ for any $u,v\in V_1$ (cf. \cite{B}). Moreover, it was proved in \cite{DM1} that $V_1$ is a reductive Lie algebra if $V$ is a strongly regular vertex operator algebra.

We  now recall the  notions of intertwining operator and fusion
rules from \cite{FHL}.
Let $M^1$, $M^2$, $M^3$ be admissible $V$-modules. An {\em intertwining
operator} $\mathcal {Y}$ of type $\left(\begin{tabular}{c}
$M^3$\\
$M^1$ $M^2$\\
\end{tabular}\right)$ is a linear map
\begin{align*}
\mathcal
{Y}: M^1&\rightarrow \Hom(M^2, M^3)\{z\},\\
 w^1&\mapsto\mathcal {Y}(w^1, z) = \sum_{n\in \mathbb{C}}{w_n^1z^{-n-1}}
\end{align*}
satisfying a number of conditions (cf. \cite{FHL}).  We use $\mathcal{I}_{M^1,M^2}^{M^3}$ to denote the vector space of intertwining operators of type $\left(\begin{tabular}{c}
$M^3$\\
$M^1$ $M^2$\\
\end{tabular}\right)$.  If $V$ is a rational vertex operator algebra and $\{M^i|0\leq i\leq p\}$ is the set of irreducible admissible $V$-modules,  we define the \emph{fusion rules} to be the formal product rules
 \begin{align*}
 M^i\times M^j=\sum_{0\leq k\leq p} N_{M^i,M^j}^{M^k}M^k,
 \end{align*}
 where $N_{M^i,M^j}^{M^k}$ denotes the dimension  of $\mathcal{I}_{M^i,M^j}^{M^k}$. The {\em fusion ring} of $V$ is defined to be the ring with $\{M^i|0\leq i\leq p\}$ as a basis and with the fusion rules as the structural constants. Let $V$ be a rational vertex operator algebra and $M$ an irreducible admissible $V$-module. If for any irreducible admissible module $M^2$,  there exists an irreducible admissible module $M^3$ such that $M\times M^2=M^3$,
 then $M$ is called a {\em simple current} module of $V$.
\subsection{Modular invariance of trace functions}
We now recall the modular invariance property of vertex operator algebra from \cite{Z}.  Let $V$ be a rational vertex operator algebra  and let $M^0, ...,M^p$ be all the irreducible $V$-modules. Then $M^i, 0\leq i \leq p$, has the form $$M^i=\bigoplus_{n=0}^{\infty}M^i_{\lambda_i+n},$$ with $M^i_{\lambda_i}\neq 0$ for some number $\lambda_i$, which is called the {\em conformal weight} of $M^i$.
 Let $\mathfrak{H} =\{\tau\in \mathbb{C}| \im\tau>0\}$ be the upper half plane. The trace function associated with $M^i$ is defined as follows: For any homogeneous element $v\in V$ and $\tau\in \mathfrak{H}$,
\begin{equation*}
Z_{M^i}(v,\tau):=\tr_{M^i}o(v)q^{L(0)-c/24}=q^{\lambda_i-c/24}\sum_{n\in\mathbb{Z}^+} \tr_{M^i_{\l_i+n}}o(v)q^n,
\end{equation*}
where $o(v)=v_{\wt v-1}$ and $q=e^{2\pi \sqrt{-1}\tau}$. Assume further that $V$ is a  $C_2$-cofinite vertex operator algebra,  then $Z_{M^i}(v,\tau)$ converges to a holomorphic function on the domain $|q| < 1$ \cite{DLM3,Z}.

Recall that the full modular group $SL(2, \mathbb{Z})$ has generators $S=\left(\begin{array}{cc}0 & -1\\ 1 & 0\end{array}\right)$, $T=\left(\begin{array}{cc}1 & 1\\ 0 & 1\end{array}\right)$ and acts on $\mathfrak{H}$ as follows:$$\gamma: \tau\longmapsto \frac{a\tau+b}{c\tau+d}, \  \gamma=\left(\begin{tabular}{cc}
$a$ $b$\\
$c$ $d$\\
\end{tabular}\right) \in SL(2, \mathbb{Z}).$$
Then the full modular group has an action on the trace functions. More precisely, we have the following result which was proved in \cite{Z} (also see \cite{DLM3}).
\begin{theorem}\label{minvariance}
 Let $V$ be a rational and $C_2$-cofinite vertex operator algebra with the irreducible $V$-modules $M^0,...,M^p.$   Then the vector space spanned by $Z_{M^0}(v,\tau),..., Z_{M^p}(v,\tau)$ is invariant under the action of $SL(2, \mathbb{Z})$ defined above, i.e., there is a representation $\rho$ of $SL(2, \mathbb{Z})$ on this vector space and the transformation matrices are independent of the choice of $v\in V$.
\end{theorem}
\subsection{Quantum dimensions}
In this subsection,  we recall some facts about quantum dimensions of irreducible modules of vertex operator algebras from \cite{DJX}. Let $V$ be a strongly regular vertex operator algebra and let $M^0=V, M^1,...,M^p$ be all the inequivalent irreducible $V$-modules. The {\em quantum dimension} of $M^i$ is defined to be
\begin{align*}
\qdim_V M^i=\lim_{y\to 0^+}\frac{Z_{M^i}(\sqrt{-1}y)}{Z_{V}(\sqrt{-1}y)},
\end{align*}
where $y$ is real and positive. The {\em global dimension}  of the vertex operator algebra $V$ is defined to be
\begin{align*}
{\rm Glob}\, V=\sum_{i=0}^p(\qdim_V M^i)^2.
\end{align*}

The following result was proved in  \cite{DJX}.
\begin{theorem}\label{qdim1}
Let $V$ be a strongly regular vertex operator algebra and let $M^0=V, M^1,...,M^p$ be all the irreducible $V$-modules. Assume further that the conformal weights of $ M^1,...,M^p$ are greater than $0$. Then \\
(1) $\qdim M^i\geq 1$ for any $0\leq i\leq p$.\\
(2) $M^i$ is a simple current $V$-module if and only if $\qdim M^i=1$.
\end{theorem}

Recall that a vertex operator algebra $U$ is called an {\em extension vertex operator algebra} of $V$ if $V$ is a vertex operator subalgebra of $U$ and $V$, $U$ have the same conformal element. 
The following result was proved in \cite{ADJR}.
\begin{theorem}\label{qdim2}
Let $U$ and  $V$ be strongly regular vertex operator algebras and let $M^0=V, M^1,...,M^p$ be all the irreducible $V$-modules. Assume further that $U$ is an extension vertex operator algebra of $V$ and that the conformal weights of $ M^1,...,M^p$ are greater than $0$. Then,  we have
\begin{align*}
{\rm Glob}\, (U)=(\qdim_VU)^2{\rm Glob}\,( V).
\end{align*}
\end{theorem}

\subsection{Modular invariants of vertex operator algebras}\label{nota1}
Next, we recall some facts about modular invariants of vertex operator algebras from \cite{DLN}. We shall assume that $V$ is a  strongly regular vertex operator algebra.
Let $M^0=V, M^1,...,M^p$ be all the irreducible $V$-modules. A {\em modular invariant} of $V$ is a $(p+1)\times (p+1)$-matrix $X$ satisfying the following conditions:\\
(M1) The entries of $X$ are nonnegative integers.\\
(M2) $X_{0,0}=1$.\\
(M3) $XS=S X$ and $XT=TX$, where we use $S, T$ to denote the modular transformation matrices $\rho(S)$ and $\rho(T)$ respectively.

In the following,  we shall construct a modular invariant of $V$ from  an extension vertex operator algebra of $V$. First,  we have the following result (cf. \cite{ABD,HKL}).
\begin{theorem}\label{rational}
Let $V$ be a strongly regular vertex operator algebra and $U$ an extension vertex operator algebra of $V$. Assume further that $U$ is simple. Then, $U$ is rational and $C_2$-cofinite.
\end{theorem}

  We now assume that $U$, $V$ are strongly regular vertex operator algebras and that $U$ is an extension vertex operator algebra of $V$.
For $u,v\in V$,  set
$$f_V(u,v,\tau_1,\tau_2)=\sum_{i=0}^pZ_{M^i}(u,\tau_1)\overline{Z_{M^j}(v,\tau_2)},$$ where $\tau_1, \tau_2 \in\mathfrak{H}$.

 Similarly, for $u,v\in U$, set
$$f_U(u,v,\tau_1,\tau_2)=\sum_{M}Z_M(u,\tau_1)\overline{Z_M(v,\tau_2)},$$
 where $M$ ranges through the equivalent classes of irreducible $U$-modules.
Since each irreducible $U$-module $M$ is a direct sum of irreducible $V$-modules, there exists a matrix $X=(X_{i,j})$ such that  $X_{i,j}\geq 0$ for all $i,j$ and for $u,v\in V$,
$$f_U(u,v,\tau_1,\tau_2)=\sum_{i,j=0}^pX_{i,j}Z_{M^i}(u,\tau_1)\overline{Z_{M^j}(v,\tau_2)}.$$
Moreover,  the matrix $X=(X_{i,j})$ is uniquely determined by the following proposition.
\begin{proposition}[\cite{DLN}]\label{unique}
 Let $M^0, ..., M^p$ be all the $V$-irreducible modules.  Set $${\bf Z}(u,\tau)=(Z_{M^0}(u,\tau), ..., Z_{M^p}(u,\tau))^T,$$
 If $A=(a_{ij})$ is a  matrix such that  for any $u,v\in V$,$${\bf Z}(u,\tau_1)^TA\overline{{\bf Z}(v,\tau_2)}=0,$$ then $A=0.$
\end{proposition}
It was further shown in \cite{DLN} that
\begin{theorem} \label{invariant}
The matrix $X$ is a modular invariant of $V$.
\end{theorem}

Furthermore, by the discussion above, we have the following simple observation.
\begin{lemma}\label{cre}
Let $V$, $U$, $X$ be as above. Suppose that there exist irreducible $V$-modules $M^i$ and $M^j$ such that $X_{i,j}\neq 0$. Then we have $X_{i,i}\neq 0$ and $X_{j,j}\neq 0$.
\end{lemma}
\subsection{Mirror extensions of vertex operator algebras}
In this subsection, we shall recall from \cite{Lin} some facts about mirror extensions of vertex operator algebras. First, we have the following result which was proved in \cite{Hu1}.
\begin{theorem}
Let $V$ be a strongly regular vertex operator algebra and let $\cc_V$ be the category of $V$-modules. Then  $\cc_V$ is a modular tensor category such that $V$ is the unit object.
\end{theorem}

Recall from \cite{BK} that for any modular tensor category $(\cc,\ot)$ and an object $M$ in $\cc$, there is a natural isomorphism
$$\theta_M: M\longrightarrow M$$  such that
\begin{align*}
\theta_{M\otimes N}&=c_{N,M}c_{M,N}(\theta_{M}\otimes \theta_{N}),\\
\theta_{1}&=\id,\\
\theta_{M^\ast}&=(\theta_M)^\ast,
\end{align*}
where $1$ denotes the unit object of $\cc$, $M^*$ denotes the dual of $M$, $c_{-,-}$ denotes the braiding of $\cc$, and  $(\theta_M)^\ast\in \Hom(M^\ast, M^\ast)$ denotes the image of $\theta_M\in \Hom(M, M)$ under the canonical map. It was shown in \cite{HKL} that an extension vertex operator algebra $U$ of $V$ induces an etale algebra $A_U$ (cf. \cite{DMNO}) in $\cc_V$ such that $\theta_{A_U}=\id$. Moreover, we have the following result which was proved in \cite{HKL} (see also \cite{Lin}).
\begin{theorem}\label{semisimple}
Let $V$ be a strongly regular vertex operator algebra and let $\cc_V$ be the category of $V$-modules. Then the following statements are equivalent\\
(1) There exists an extension vertex operator algebra $U$ of $V$ such that $U$ viewed as a $U$-module is irreducible.\\
(2) There exists an etale algebra $A_U$ in $\cc_V$ such that $A_U$ viewed as a $V$-module is isomorphic to $U$ and that $\theta_{A_U}=\id$.
\end{theorem}
 We now let $(U, Y, \1, \w)$ be a vertex operator algebra and let
$(V, Y, \1, \w')$ be a vertex operator subalgebra of $U$ such
that $\w'\in U_2$ and $L(1)\w'=0$. It was proved in \cite{FZ} that  $(V^c, Y, \1, \w-\w')$ is also a vertex operator subalgebra of
$U$, where $V^c={\rm C}_U(V)=\{v\in U|\w'_0v=0\}$.
Assume further that $(V, Y, \1, \w)$, $(V^c, Y, \1, \w-\w')$,
$(U, Y, \1, \w')$ are strongly regular
 and $(V^c)^c=V$,  and denote the tensor products of the module categories $\cc_V$, $\cc_{V^c}$ and $\cc_{V\ot V^c}$ by $\bt_V$, $\bt_{V^c}$, $\bt_{V\ot V^c}$, respectively. Then we have the following results which were proved in \cite{Lin}.
\begin{theorem}\label{mirror}
(1) As a $V\ot V^c$-module,  $U$ has the following decomposition
$$U=V\ot V^c\oplus (\oplus_{i=1}^nM^i\ot N^i), $$
where  $M^0=V, M^1,\cdots, M^n$ (resp. $N^0=V^c, N^1,\cdots, N^n$) are mutually inequivalent irreducible $V$-modules (resp. $V^c$-modules).\\
(2) Let $\ck (\cc_V)$ and $\ck(\cc_{V^c})$ be the Grothendieck rings of $\cc_V$ and $\cc_{V^c}$, respectively. Then $\Z M^0\oplus  \cdots \oplus\Z M^n$ (resp. $\Z N^0\oplus \cdots \oplus\Z N^n$) forms a subring of $\ck (\cc_V)$ (resp. $\ck (\cc_V^c)$).
(3) For any $0\leq i_1, i_2,i_3\leq n$,  $N_{M^{i_1}, M^{i_2}}^{M^{i_3}}=N_{N^{i_1}, N^{i_2}}^{N^{i_3}}$.
\end{theorem}

Let $\cc_V^0$ (resp. $\cc_{V^c}^0$) be the tensor subcategory of $\cc_V$ (resp. $\cc_{V^c}$) such that the Grothendieck ring of $\cc_V^0$ (resp. $\cc_{V^c}^0$) is isomorphic to the subring $\Z M^0\oplus\cdots \oplus \Z M^n$ (resp. $\Z N^0\oplus\cdots \oplus \Z N^n$) of $\ck (\cc_V)$ (resp. $\ck (\cc_{V^c})$). Then we have the following results which were also  proved in \cite{Lin}.
\begin{theorem}\label{mirror1}
(1) There is a braid-reversing equivalence $\ct:\cc_V^0\to \cc_{V^c}^0$ such that $\ct(M^i)\cong(N^{i})'$.\\
(2) If there is a vertex operator algebra structure $Y_{V^e}(.,z)$ on the $V$-module $$V^e=V\oplus (\oplus_{i=1}^nm_iM^i),$$  where $m_i$'s are nonnegative integers, such that $(V^e, Y_{V^e}(.,z))$ is an extension vertex operator algebra of $V$, then there exists a vertex operator algebra structure $Y_{(V^c)^e}(.,z)$ on the $V^c$-module $$(V^c)^e=V^c\oplus (\oplus_{i=1}^nm_i(N^{i})'),$$ such that $((V^c)^e, Y_{(V^c)^e}(.,z))$ is an extension vertex operator algebra of $V^c$. Moreover, $(V^c)^e$ is a simple vertex operator algebra if $V^e$ is a simple vertex operator algebra.
\end{theorem}
Furthermore, we have the following result about the uniqueness.
\begin{theorem}\label{mirror2}
Let $V^e$ and $(V^c)^e$ be as above. Suppose that for any two extension vertex operator algebras $(V^e, Y_{V^e}^1(\cdot, z))$, $(V^e, Y_{V^e}^2(\cdot, z))$ of $V$, there exists a linear isomorphism $\phi:V^e\to V^e$ satisfying the following conditions
 \begin{align*}
 \phi|_V=\id, \ \ \phi(Y_{V^e}^1(u^1, z)u^2)=Y_{V^e}^2(\phi(u^1), z)\phi(u^2),\ {\rm for\ any }\ u^1, u^2\in V^e.
 \end{align*}
 Then for any two extension vertex operator algebras  $((V^c)^e, Y_{(V^c)^e}^1(\cdot, z))$, $((V^c)^e, Y_{(V^c)^e}^2(\cdot, z))$ of $V^c$, there exists a linear isomorphism $\phi^c:(V^c)^e\to (V^c)^e$ satisfying the following conditions
 \begin{align*}
 \phi^c|_{V^c}=\id, \ \ \phi^c(Y_{(V^c)^e}^1(v^1, z)v^2)=Y_{(V^c)^e}^2(\phi^c(v^1), z)\phi^c(v^2),\ {\rm for\ any }\ v^1, v^2\in (V^c)^e.
 \end{align*}
\end{theorem}

\pf Assume that there exist two vertex operator algebras $((V^c)^e, Y_{(V^c)^e}^1(\cdot, z))$ and  $((V^c)^e, Y_{(V^c)^e}^2(\cdot, z))$ such that $((V^c)^e, Y_{(V^c)^e}^1(\cdot, z))$,  $((V^c)^e, Y_{(V^c)^e}^2(\cdot, z))$ are extension vertex operator algebras of $V^c$. Let $A_{(V^c)^e}^1$, $A_{(V^c)^e}^2$ be the etale algebras in $\cc_{V^c}^0$ induced from $((V^c)^e, Y_{(V^c)^e}^1(\cdot, z))$,  $((V^c)^e, Y_{(V^c)^e}^2(\cdot, z))$, respectively. Since $\ct:\cc_V^0\to \cc_{V^c}^0$ is a braid-reversing equivalence, there exists a braid-reversing functor $\cg:\cc_{V^c}^0\to \cc_V^0$ such that $\ct\circ \cg$, $\cg\circ \ct$ are natural isomorphic to $\id_{\cc_{V^c}^0}$, $\id_{\cc_V^0}$, respectively, that is, there exist a family of isomorphisms $\eta^1(N):\ct\circ\cg(N)\to N$, $N\in \cc_{V^c}^0 $, and $\eta^2(M):\cg\circ \ct(M)\to M$, $M\in \cc_V^0$ satisfying
\begin{align*}
\eta^1(N^2)\circ (\ct\circ \cg(g))=g\circ \eta^1(N^1),\ \ \eta^2(M^2)\circ (\cg\circ \ct(f))=f\circ \eta^2(M^1),
\end{align*} for any $M^1, M^2\in \cc_{V}^0 $, $N^1, N^2\in \cc_{V^c}^0 $ and $f:M^1\to M^2$, $g: N^1\to N^2$ (see \cite{Kas}). Note that, by Theorem \ref{mirror1}, $\cg(V^c)$ is isomorphic to $V$ and that $\cg(A_{(V^c)^e}^1)$, $\cg(A_{(V^c)^e}^2)$ are two etale algebras in $\cc_{V}^0$ such that $\cg(A_{(V^c)^e}^1)$, $\cg(A_{(V^c)^e}^2)$ viewed as $V$-modules are isomorphic to $V^e$. It follows from Theorem \ref{semisimple} that there exist two vertex operator algebra structures $(\cg((V^c)^e), Y^1(\cdot, z))$,  $(\cg((V^c)^e), Y^2(\cdot, z))$ such that $(\cg((V^c)^e), Y^1(\cdot, z))$,  $(\cg((V^c)^e), Y^2(\cdot, z))$ are extension vertex operator algebras of $\cg(V^c)$. By assumption, there exists a linear isomorphism $\phi:\cg((V^c)^e)\to \cg((V^c)^e)$ satisfying the following conditions
 \begin{align*}
 \phi|_{\cg(V^c)}=\id, \ \ \phi(Y^1(u^1, z)u^2)=Y^2(\phi(u^1), z)\phi(u^2),\ {\rm for\ any }\ u^1, u^2\in \cg((V^c)^e).
 \end{align*}
Then we know that $\phi$ induces an etale algebra isomorphism $\tilde\phi:\cg(A_{(V^c)^e}^1)\to \cg(A_{(V^c)^e}^2)$ such that $\tilde\phi|_{\cg(V^c)}=\id$. Hence, $\ct(\tilde \phi)$ is an etale algebra isomorphism from $\ct\circ\cg(A_{(V^c)^e}^1)$ to $\ct\circ\cg(A_{(V^c)^e}^2)$ such that $\ct(\tilde \phi)|_{\ct(\cg(V^c))}=\id$. Note that $\eta^1(A_{(V^c)^e}^1), \eta^1(A_{(V^c)^e}^2)$ are algebra isomorphism (see Definition XI.4.1 of \cite{Kas}). As a result, $\eta^1(A_{(V^c)^e}^2)\circ\ct(\tilde \phi)\circ\eta^1(A_{(V^c)^e}^1)^{-1}$ is an etale algebra isomorphism from $A_{(V^c)^e}^1$ to $A_{(V^c)^e}^2$ such that $\eta^1(A_{(V^c)^e}^2)\circ\ct(\tilde \phi)\circ\eta^1(A_{(V^c)^e}^1)^{-1}|_{V^c}=\id$. Therefore, $\eta^1(A_{(V^c)^e}^2)\circ\ct(\tilde \phi)\circ\eta^1(A_{(V^c)^e}^1)^{-1}$ will induce  the desired isomorphism.  The proof is complete.
\qed

As a corollary, we have the following:
\begin{corollary}
Let $V^e$ and $(V^c)^e$ be as above. Suppose that there is a unique vertex operator algebra structure on $V^e$ as an extension vertex operator algebra of $V$. Then there is a unique vertex operator algebra structure on $(V^c)^e$ as an extension vertex operator algebra of $V^c$.
\end{corollary}

\section{Mirror extensions of affine vertex operator algebra $L_{sl_9}(3\lddot_0)$}
\def\theequation{3.\arabic{equation}}
\setcounter{equation}{0}
\subsection{Affine vertex operator algebras}  In this subsection, we shall recall some facts about affine vertex operator algebras from \cite{FZ} and \cite{LL}. Let $\g$ be a finite dimensional simple Lie algebra and $\langle\, ,\, \rangle$ the normalized Killing form of $\g$, i.e., $\langle\theta, \theta\rangle=2$ for the highest root $\theta$ of $\g$. Fix a Cartan subalgebra $\h$ of $\g$ and  denote the corresponding root system by $\Delta_{\g}$ and the root lattice by $Q$. We further fix simple roots $\{\alpha_1,\cdots,\alpha_l\}$, and denote the set of positive roots by $\Delta_{\g}^+$. Then the weight lattice $P$ of $\g$ is the set of $\lambda\in \h$ such that $\frac{2\langle\lambda, \alpha\rangle}{\langle\alpha, \alpha\rangle}\in\Z$ for all $\alpha\in \Delta_{\g}$. Note that $P$ is equal to $\oplus_{i=1}^l\Z\Lambda_i$, where $\Lambda_i$ are the fundamental weights defined by the equation $\frac{2\langle\Lambda_i, \alpha_j\rangle}{\langle\alpha_j, \alpha_j\rangle}=\delta_{i,j}$. We also use the standard notation $P_+$ to denote the set of dominant weights $\{\Lambda\in P\mid\frac{2\langle\Lambda, \alpha_j\rangle}{\langle\alpha_j, \alpha_j\rangle}\geq 0,~1\leq j\leq l \}$. For any  $\alpha\in \Delta_{\g}^+$, we fix $x_{\pm\alpha}\in \g_{\pm\alpha}$ such that $[x_{\alpha},x_{-\alpha}]=h_{\alpha}$, $[h_{\alpha},x_{\pm\alpha}]=\pm 2 x_{\pm\alpha}$, where $h_{\alpha}=\frac{2}{\langle\alpha,\alpha\rangle}\alpha$. 

Recall that the affine Lie algebra associated to $\g$ is defined on $\hat{\g}=\g\otimes \C[t^{-1}, t]\oplus \C K$ with Lie brackets
\begin{align*}
[x(m), y(n)]&=[x, y](m+n)+\langle x, y\rangle m\delta_{m+n,0}K,\\
[K, \hat\g]&=0,
\end{align*}
for $x, y\in \g$ and $m,n \in \Z$, where $x(n)$ denotes $x\otimes t^n$. In particular, $\hat{\h}=\h\otimes \C[t^{-1}, t]\oplus \C K$ is a subalgebra of $\hat \g$.

For a positive integer $k$ and a weight $\Lambda \in P$, let $L_{\g}(\Lambda)$ be the irreducible highest weight module for $\g$ with highest weight $\Lambda$ and define
\begin{align*}
V_{\g}(k, \Lambda)=\Ind_{\g\otimes \C[t]\oplus \C K}^{\hat \g}L_{\g}(\Lambda),
\end{align*}
where $L_{\g}(\Lambda)$ is viewed as a module for $\g\otimes \C[t]\oplus \C K$ such that $\g\otimes t\C[t]$ acts as $0$ and $K$ acts as $k$. It is well-known that $V_{\g}(k, \Lambda)$ has a unique maximal proper submodule which is denoted by $J(k, \Lambda)$ (see \cite{K}). Let $L_{\g}(k, \Lambda)$ be the corresponding irreducible quotient module. It was proved in \cite{FZ} that $L_{\g}(k, 0)$ has a vertex operator algebra structure such that the conformal element
\begin{align*}\label{virasoro}
\w=\frac{1}{2(k+h^{\vee})}\big( \sum_{i=1}^{\dim \h} u_i(-1)u_i(-1)\1+\sum_{\alpha\in \Delta_{\g}}\frac{\langle \alpha, \alpha\rangle}{2}x_\alpha(-1)x_{-\alpha}(-1)\1\big),
\end{align*}
where $h^\vee$ denotes the dual Coxeter number of $\g$ and $\{u_i|1\leq i\leq \dim \h\}$ is an orthonormal basis of $\h$ with respect to $\langle\,  ,\,\rangle$.
\begin{theorem}[\cite{FZ,K}]
Let $k$ be a positive integer. Then \\
(1) $L_{\g}(k, 0)$ is a strongly regular vertex operator algebra;\\
 (2) $L_{\g}(k, \Lambda)$ is a module for the vertex operator algebra $L_{\g}(k, 0)$ if and only if $\Lambda \in P_+^k$, where $P_+^k=\{\Lambda \in P_+|\langle\Lambda, \theta\rangle\leq k\}$;\\
 (3) If $L_{\g}(k, \Lambda)$ is an $L_{\g}(k, 0)$-module  such that $L_{\g}(k, \Lambda)\ncong L_{\g}(k, 0)$, then the conformal weight of $L_{\g}(k, \Lambda)$ is positive.
\end{theorem}

We next recall some facts about simple current modules of $L_{\g}(k, 0)$. Let $\theta=\sum_{i=1}^l a_i\alpha_i$, $a_i\in \Z_+$, be the highest root. It is well-known that the irreducible $L_{\g}(k, 0)$-module $L_{\g}(k, k\Lambda_i)$ is a simple current $L_{\g}(k, 0)$-module if $a_i=1$ (see \cite{DLM1,F,FG,L3}). In particular, for $\g=sl_{n+1}$, $L_{\g}(k, k\Lambda_1),\cdots,L_{\g}(k, k\Lambda_n)$  are simple current $L_{sl_{n+1}}(k, 0)$-modules. Moreover, these are all the simple current $L_{sl_{n+1}}(k, 0)$-modules (see \cite{DLM1}, \cite{L3}).
\subsection{Mirror extensions of affine vertex operator algebra $L_{sl_9}(3,0)$}\label{nota}
In this subsection, we shall construct some extension vertex operator algebra of $L_{sl_9}(3,0)$. Consider the affine vertex operator algebra $L_{sl_{27}}(1,0)$. It is well-known that $L_{sl_{27}}(1,0)$ contains a vertex operator subalgebra isomorphic to $L_{sl_3}(9,0)\ot L_{sl_9}(3,0)$. To determine the decomposition of $L_{sl_{27}}(1,0)$ viewed as an $L_{sl_3}(9,0)\ot L_{sl_9}(3,0)$-module, we need to recall some notations from \cite{OS}. Let $\lambda = (\lambda_1\geq \cdots\geq \lambda_k > 0 = \lambda_{k+1}=\cdots )$ be a partition of
$|\lambda|= \lambda_1 + \cdots + \lambda_k$. We define $h(\lambda)= k$ and identify $\lambda$ with its
corresponding Young diagram; thus $h(\lambda)$ is just the number of rows in this diagram. We
write $I_n$ for the set of all partitions with $h(\lambda)\leq n$. Let $I_{n,m}$ be the set of all $\lambda\in  I_n$ with
$\lambda_1\leq m$. Hence $\lambda\in I_{n,m}$ if and only if its Young diagram fits into an $m\times n$ rectangle.
Denote by $\lambda^t$ the transposed partition of $\lambda$. Clearly, $\lambda\in I_{n,m}$ implies $\lambda^t\in I_{m,n}$.

Let $C_{n,m}= \{(a_0, a_1, \cdots, a_{n-1})\in \N^n\mid a_0+\cdots +a_{n-1} = m\}$. By identifying $a=(a_0, a_1, \cdots, a_{n-1})$ with $\hat{a}=a_1\Lambda_1+ \cdots+a_{n-1}\Lambda_{n-1}$, we know that $C_{n,m}$ is exactly the set of dominant $\hat {sl}_n$-weights of level $m$. For any $\lambda\in I_{n,m}$, we define
$$w_{n,m}(\lambda)= (m-\lambda_1 + \lambda_n, \lambda_1- \lambda_2, \lambda_2-\lambda_3, \cdots, \lambda_{n-1} -\lambda_n)\in C_{n,m}.$$ Conversely, we define $d_{n,m}: C_{n,m}\to I_{n,m}$ by sending
$(a_0, a_1, \cdots, a_{n-1})$ to partition $(a_1+\cdots +a_{n-1}, a_2+\cdots +a_{n-1}, \cdots, a_{n-1}, 0,\cdots)$.

Note that
$|d_{n,m}(a_0, a_1, \cdots, a_{n-1})|=\sum_i ia_i$.
For $a=(a_0, a_1, \cdots, a_{n-1})\in C_{n,m}$, we say that $|d_{n,m}(a)|+n\Z\in \Z/n\Z$ is the degree of $a$ and
write it $\deg(a)$. In particular, we will consider the subset
$C^0_{n,m}=\{a\in C_{n,m}| \deg(a) = 0 ~ ({\rm mod}\  n)\}$. Let $\rho_n : C_{n,m}\to C_{n,m}$ be the cyclic permutation $\rho_n(a_0, a_1, \cdots, a_{n-1}) =(a_{n-1}, a_0,  \cdots, a_{n-2})$. We also define $\tau: C^0_{n,m}\to C^0_{m,n}$ by
\begin{align*}
\tau(a)=\rho_m^{\frac{-|d_{n,m}(a)|}{n}}(w_{m,n}(d_{n,m}(a)^t)).
\end{align*}
Then we have the following result which was proved in \cite{OS}.
\begin{theorem}\label{decom}
Let $\ldot_1, \ldot_2$ be the fundamental weights of $sl_3$ and let $\lddot_1,\cdots, \lddot_8$ be the fundamental weights of $sl_9$. Then the decomposition of $L_{sl_{27}}(1,0)$ viewed as an $L_{sl_3}(9,0)\ot L_{sl_9}(3,0)$-module is as follows:
\begin{align*}
L_{sl_{27}}(1,0)=\oplus_{a\in C^0_{3,9}}L_{sl_3}(9, \hat{a})\ot L_{sl_9}(3, \hat{\tau(a)}),
\end{align*}
where $\hat{a}$ denotes $a_1\ldot_1+a_2\ldot_2$ and $\hat{\tau(a)}$ denotes ${\tau(a)}_1\lddot_1+\cdots +{\tau(a)}_8\lddot_8$.
\end{theorem}

On the other hand, it is well-known that $L_{sl_3}(9, 0)$ has an extension vertex operator algebra $L_{E_6}(1,0)$ (see \cite{AL,KS}). Moreover, it was shown in \cite{KS} that the vertex operator algebra $L_{E_6}(1,0)$ viewed as an $L_{sl_3}(9, 0)$-module has the following decomposition
\begin{align*}
L_{E_6}(1,0)=&L_{sl_3}(9, 0)\oplus L_{sl_3}(9, 9\ldot_1)\oplus L_{sl_3}(9, 9\ldot_2)\oplus L_{sl_3}(9, \ldot_1+4\ldot_2)\\
&\ \ \oplus L_{sl_3}(9, 4\ldot_1+\ldot_2)\oplus L_{sl_3}(9, 4\ldot_1+4\ldot_2).
\end{align*}
Furthermore, we have the following result which is a slight generalization of Theorem 3.8 of \cite{AL}.
\begin{theorem}\label{exten}
Let $(U^1, Y_1(\cdot, z))$, $(U^2, Y_2(\cdot, z))$ be extension vertex operator algebras of $L_{sl_3}(9,0)$ such that $(U^1, Y_1(\cdot, z))$, $(U^2, Y_2(\cdot, z))$ are strongly regular and that $U^1$, $U^2$ viewed as  modules of $L_{sl_3}(9,0)$ have the following decomposition
\begin{align*}
\begin{split}
U^1\cong U^2 \cong& L_{sl_3}(9,0)\oplus L_{sl_3}(9,9\Lambda_2) \oplus L_{sl_3}(9,9\Lambda_1) \oplus L_{sl_3}(9,\Lambda_1+4\Lambda_2)\\
&\oplus L_{sl_3}(9,4\Lambda_1+\Lambda_2)\oplus L_{sl_3}(9,4\Lambda_1+4\Lambda_2).
\end{split}
\end{align*}
Then there exists an isomorphism $\phi:U^1\to U^2$ such that
\begin{align*}
 \phi|_{L_{sl_3}(9,0)}=\id, \ \ \phi(Y_1(u^1, z)u^2)=Y_2(\phi(u^1), z)\phi(u^2),\ {\rm for\ any }\ u^1, u^2\in U^1.
 \end{align*}
\end{theorem}
\pf  By assumption, $(U^1, Y_1(\cdot, z))$ is an extension vertex operator algebra of $L_{sl_3}(9,0)$ such that $(U^1, Y_1(\cdot, z))$ is strongly regular and that $U^1$ viewed as a module of $L_{sl_3}(9,0)$ has the following decomposition
\begin{align*}
\begin{split}
U^1 \cong& L_{sl_3}(9,0)\oplus L_{sl_3}(9,9\Lambda_2) \oplus L_{sl_3}(9,9\Lambda_1) \oplus L_{sl_3}(9,\Lambda_1+4\Lambda_2)\\
&\oplus L_{sl_3}(9,4\Lambda_1+\Lambda_2)\oplus L_{sl_3}(9,4\Lambda_1+4\Lambda_2).
\end{split}
\end{align*}
It follows from Theorem 3.8 of \cite{AL} that there exists an isomorphism $\psi_1: U^1\to L_{E_6}(1,0)$ such that
\begin{align*}
\psi_1(Y_1(u,z)v)=Y_{L_{E_6}(1,0)}(\psi_1(u),z)\psi_1(v)
\end{align*}  for any $u, v\in U^1$, where $Y_{L_{E_6}(1,0)}(\cdot,z)$ denotes the vertex operator map of $L_{E_6}(1,0)$. In particular, $\psi_1|_{L_{sl_3}(9,0)}$ is an automorphism of $L_{sl_3}(9,0)$.

We next show that there exists an automorphism $\psi_2$ of $L_{E_6}(1,0)$ such that $\psi_2\circ \psi_1|_{L_{sl_3}(9,0)}=\id$. It is good enough to show that any automorphism of $L_{sl_3}(9,0)$ can be lifted to an automorphism of $L_{E_6}(1,0)$.  Note that  any automorphism of $sl_3$ has the form
$\varphi\exp(2\pi \sqrt{-1}h)$,
where $\varphi$ denotes the diagram automorphism of $sl_3$ and $h$ is an element of a Cartan subalgebra of $sl_3$ (see \cite{L4}). Moreover, it was shown in subsection 2.5 of \cite{M} that the diagram automorphism $\varphi$ of $sl_3$ can be lifted to an automorphism of $E_6$. Since for any $h$, $\exp(2\pi \sqrt{-1}h)$ is also an automorphism of $E_6$, we then know that any automorphism of $sl_3$ can be lifted to an automorphism of $E_6$. It follows that  any automorphism of $L_{sl_3}(9,0)$ can be lifted to an automorphism of $L_{E_6}(1,0)$.

Similarly, there exists a vertex operator algebra isomorphism $\phi_2:U^2\to L_{E_6}(1,0)$ such that $\phi_2|_{L_{sl_3}(9,0)}=\id$. Hence, $\phi_2^{-1}\circ\psi_2\circ \psi_1$ is the desired isomorphism. The proof is complete.   \qed

The next lemma can be proved by a direct calculation.
\begin{lemma}\label{module}
Let $\tau$ be the map defined as above. Then we have
\begin{align*}
\tau((9,0,0))=(3,0,0,0,0,0,0,0,0),\\
\tau((0,9,0))=(0,0,0,3,0,0,0,0,0),\\
\tau((0,0,9))=(0,0,0,0,0,0,3,0,0),\\
\tau((4,4,1))=(0,0,0,1,0,0,0,1,1),\\
\tau((4,1,4))=(0,1,1,0,0,0,1,0,0),\\
\tau((1,4,4))=(1,0,0,0,1,1,0,0,0).
\end{align*}
\end{lemma}
Since the affine vertex operator algebra $L_{E_6}(1,0)$ is self-dual, we then have the following result by Theorems \ref{mirror1}, \ref{decom} and Lemma \ref{module}.
\begin{theorem}\label{exist}
There is a vertex operator algebra structure on
\begin{align*}
\widetilde{L_{sl_9}(3, 0)}=&L_{sl_9}(3, 0)\oplus L_{sl_9}(3, 3\lddot_3)\oplus L_{sl_9}(3, 3\lddot_6)\oplus L_{sl_9}(3, \lddot_1+\lddot_2+\lddot_6)\\
&\ \ \oplus L_{sl_9}(3, \lddot_3+\lddot_7+\lddot_8)\oplus L_{sl_9}(3, \lddot_4+\lddot_5)
\end{align*}
such that $\widetilde{L_{sl_9}(3, 0)}$ is an extension vertex operator algebra of  $L_{sl_9}(3, 0)$ and strongly regular.
\end{theorem}
Moreover, by Theorems \ref{mirror2} and \ref{exten}, we have:
\begin{theorem}\label{unique1}
  Let $(U^1,Y_1(\cdot, z))$ and $(U^2, Y_2(\cdot, z))$ be two strongly regular vertex operator algebras satisfying the following conditions:\\
(1) $U^1$ and $U^2$ are extension vertex operator algebras of $L_{sl_9}(3,0)$.\\
(2) $U^1$ and $U^2$ viewed as $L_{sl_9}(3,0)$-modules are isomorphic to
\begin{align*}
&L_{sl_9}(3, 0)\oplus L_{sl_9}(3, 3\lddot_3)\oplus L_{sl_9}(3, 3\lddot_6)\oplus L_{sl_9}(3, \lddot_1+\lddot_2+\lddot_6)\\
&\ \ \oplus L_{sl_9}(3, \lddot_3+\lddot_7+\lddot_8)\oplus L_{sl_9}(3, \lddot_4+\lddot_5).
\end{align*}
Then there exists a vertex operator algebra isomorphism $\phi^c:U^1\to U^2$ such that $\phi^c|_{L_{sl_9}(3,0)}=\id$.
\end{theorem}
\subsection{Lattice vertex operator algebras }\label{sec:3.3}
In this subsection, we recall some facts about lattice vertex operator
algebras from  \cite{B}, \cite{FLM} and \cite{LL}.  Let $L$ be a positive definite even lattice. We denote the $\Z$-bilinear form on $L$ by $\la\, ,\,\ra$. There is a canonical $\Z$-bilinear form $c_0$ on $L$ defined as follows:
\begin{align*}
c_0: &L\times L\to \Z/2\Z\\
&(\alpha, \beta)\mapsto \la \alpha, \beta\ra+2\Z.
\end{align*}
Since $L$ is an even lattice, the $\Z$-bilinear form $c_0$ is alternating. Thus there is
a central extension $\hat{L}$  of $L$ by the cyclic group $\langle\kappa\rangle$ of order $2$ with generator $\kappa$, that is,
$$1\to \la\kappa\ra\to \hat{L}\stackrel{-}{\to} L\to 1,$$ such that the corresponding commutator map is $c_0$ (see \cite{FLM}).
We choose a section $e: L\to \hat{L}$
such that $e_0 = 1$ and that the corresponding 2-cocycle $\epsilon_0: L\times L\to \Z/2\Z$, which is defined by $e_{\alpha}e_{\beta}=\kappa^{\epsilon_0(\alpha, \beta)}e_{\alpha+\beta}$ for
$ \alpha, \beta \in L$, is a $\Z$-bilinear form satisfying the following condition:
\begin{align*}
\epsilon_0(\alpha, \alpha)=\frac{1}{2}\la\alpha,\alpha\ra.
\end{align*}
Hence,  we have $\epsilon_0(\alpha,
\beta)-\epsilon_0(\beta, \alpha)=c_0(\alpha, \beta)$  for
$ \alpha, \beta \in L$ (see \cite{FLM}).

Set $\h=\C\otimes_{\Z}L$ and extend the $\Z$-bilinear form on $L$ to $\h$ by $\C$-linearity. The corresponding affine Lie algebra is $\hat{\h}=\h\ot \C[t, t^{-1}]\oplus \C c$ with Lie brackets
\begin{align*}
&[x(m), y(n)]=\langle x, y\rangle m\delta_{m+n,0}c,\\
&[c, \hat\h]=0,
\end{align*}
for $x, y\in \h$ and $m,n \in \Z$, where $x(n)$ denotes $x\otimes t^n$. Set $$\hat{\h}^-=\h\ot t^{-1}\C[t^{-1}].$$ Hence, $\hat{\h}^-$ is an abelian subalgebra of $\hat{\h}$. We then consider the induced $\hat{\h}$-module
$$M(1)=U(\hat{\h})\ot_{U(\C[t]\ot \h\oplus \C c)}\C\cong S(\hat{\h}^-)\ \ \text{    (linearly)},$$
where $U(.)$ denotes the universal enveloping algebra and $\C[t]\ot \h$ acts trivially on $\C$, $c$ acts on $\C$ as multiplication by $1$.

Consider the
$\hat{L}$-module
\begin{align*}
\C\{L\}=\C[\hat L]/\C[\hat L](\kappa+1),
 \end{align*}
 where $\C[.]$ denotes the group algebra.  For $a\in \hat L$, we use $\iota(a)$ to denote the image of $a$ in $\C\{L\}$. Then the action of $\hat L$ on $\C\{L\}$ is given by
\begin{align*}
&a\cdot \iota(b)=\iota(ab),\ \ \kappa\cdot \iota(b)=-\iota(b)
\end{align*}
for $a, b\in \hat L$.
 For a formal variable $z$ and an element $h
\in \h$, we define an operator $h(0)$ on
$\C\{L\}$ by $h(0)\cdot \iota(a)=\la h, \bar a\ra \iota(a)$ and an action $z^h$ on $\C\{L\}$ by $z^{h}\cdot \iota(a)=z^{\la h,
\bar a\ra}\iota(a)$.

 Set
$$V_L=M(1)\otimes_{\C}\C\{L\}.$$Then $\hat{L}$, $h(n)(n\neq 0)$, $h(0)$ and $z^{h}$ act naturally on $V_L$ by acting on either $M(1)$ or $\C\{L\}$ as indicated above. Denote $\iota(1)$ by $\1$ and set $$\w=\frac{1}{2}\sum_{i=1}^{d}h_i(-1)^2\1,$$ where $h_1, ..., h_d$ is an orthonormal basis of $\h$. Then we know
that
$(V_L, Y(., z), \1, \w)$ has a vertex operator algebra structure (see \cite{B}, \cite{FLM}), the vertex operator $Y(.,z)$ is determined by
$$Y(h(-1)\1, z)=h(z)=\sum_{n\in \Z}h(n)z^{-n-1}\ \ (h\in \h),$$
$$Y(a, z)=E^-(-\bar a, z)E^+(-\bar a, z)az^{\bar a}\ \ (a\in \hat L),$$
where
\begin{align*}
E^-(\bar a, z)=\exp(\sum_{n<0}\frac{\bar a(n)}{n}z^{-n}),\ \
 E^+(\bar a, z)=\exp(\sum_{n>0}\frac{\bar a(n)}{n}z^{-n}).
\end{align*}
\vskip0.25cm

\subsection{Automorphisms of mirror extension $\widetilde{L_{sl_9}(3, 0)}$ }\label{a83inA83}
We now consider the root lattice of type $A_8$ and let $L$ be the positive definite even lattice isomorphic to $A_8\oplus A_8\oplus A_8$.
Let $\eta_i$, $i=1,2,3$, be the natural inclusion of $A_8$ into the $i$-th summand of $L=A_8\oplus A_8\oplus A_8$.

Let $(A_8)_2=\{\alpha\in A_8\mid \langle\alpha,\alpha\rangle=2\}$.
Set $\tilde h=(\eta_1+\eta_2+\eta_3)(h)(-1)\cdot\1\in V_L$ for any $h\in A_8\ot _{\Z}\C$ and $$E_\alpha=\iota(e_{\eta_1(\alpha)})+\iota(e_{\eta_2(\alpha)})+\iota(e_{\eta_3(\alpha)})\in V_L\qquad  \text{ for }\alpha\in (A_8)_2. $$
Then the vertex operator subalgebra $\la S \ra$ of $V_L$ generated by
$$ S=\{\tilde h|h\in A_8\otimes_{\Z}\C\}\cup \{E_\alpha|\alpha \in A_8, \langle \alpha,\alpha\rangle=2\}$$
is  isomorphic to the affine VOA $L_{sl_9}(3, 0)$ (see \cite{DL, FLM}). Moreover, it was proved in \cite{FLM} that $\{\tilde h|h\in A_8\otimes_{\Z}\C\}$ and $ \{E_\alpha\mid \alpha \in (A_8)_2 \}$ satisfy the following relations
\begin{align*}
&[\tilde h, E_\alpha]=\la h,\alpha\ra E_{\alpha},\\
&[E_\alpha,E_{-\alpha}]=(-1)^{\epsilon_0(\alpha,-\alpha)}\tilde\alpha,\\
&[E_\alpha,E_\beta]=(-1)^{\epsilon_0(\alpha,\beta)} E_{\alpha+\beta},\ \ {\rm if }\ \alpha+\beta\in (A_8)_2,\\
&[E_\alpha,E_\beta]=0,\ \ {\rm otherwise}.
\end{align*}

Let $\Omega$ be the conformal element of $L_{{sl}_9}(3,0)$. By the Sugawara construction, it is given by
\[
\Omega= \frac{1}{2(3+9)} \left[ \sum_{k=1}^8 \tilde{h^k} +
\sum_{\al\in (A_8)_2} (E_\al)_{-1}(-E_{-\al})\right],
\]
where $\{h^1, \dots, h^8\}$ is an orthonormal basis of $A_8\otimes_{\Z}\C$. Note that the dual vector of $ E_ \alpha$ is $-E_{ - \alpha}$ since $\epsilon_0(\alpha,\alpha)=1$.

Let $E= \{(\alpha,\alpha,\alpha)\mid \alpha\in A_8\}$ and $P = \{(\alpha,\beta,\gamma)\mid \alpha, \beta,\gamma \in A_8, \alpha +\beta+\gamma=0\}$.  Then, by a direct calculation (see \cite{CL}),  we
have
\[
\begin{split}
\Omega&=\omega_E +\frac{3}4 \omega_{P} -\frac{1}{12} \sum_{\al\in (A_8)_2\atop 1\leq i<j\leq 3} e_{\eta_i(\al)-\eta_j(\al)},
\end{split}
\]
where $\omega_S$ denotes the conformal element of the lattice VOA $V_S$.

\medskip

Let $\theta: \hat L\to \hat L$ be the automorphism of $\hat L$ defined by $\theta(a)=\kappa^{\epsilon_0(\bar{a}, \bar{a})}a^{-1}$. It was shown in \cite{FLM} that $\theta$ induces an automorphism of $V_L$ which is also denoted by $\theta$. Explicitly,
 $\theta:V_L\to V_L$ is the linear map defined by
\begin{align*}
\theta(\alpha_1(-n_1)\cdots\alpha_k(-n_k)\otimes \iota(a))=(-1)^k\alpha_1(-n_1)\cdots\alpha_k(-n_k)\otimes \iota(\theta(a)).
\end{align*}
In particular, we have $\theta(\iota(e_{(\alpha,0, 0)}))=-\iota(e_{(-\alpha,0, 0)})$, $\theta(\iota(e_{(0,\alpha,0)}))=-\iota(e_{(0,-\alpha,0)})$ and $\theta(\iota(e_{(0, 0, \alpha)}))=-\iota(e_{(0,0, -\alpha)})$ for any $\alpha\in (A_8)_2$. Therefore, we have $\theta(\la S \ra)=\la S \ra$, that is, $\theta|_{\la S \ra}$ induces an automorphism of $L_{sl_9}(3, 0)$, which is also denoted by $\theta$.
\vskip0.25cm

As a key result of this section, we shall show that the automorphism $\theta$ of $L_{sl_9}(3, 0)$ may be lifted to an automorphism of $\widetilde{ L_{sl_9}(3, 0)}$.

First, we have the following general result which is a slight generalization of Proposition 3.2 of \cite{S}.
\begin{theorem}\label{lift1}
Let $V$ be a strongly regular vertex operator algebra. Let $g$ be an automorphism of $V$ and $U$ an extension vertex operator algebra of $V$. Assume further  that $V$, $g$ and  $U$ satisfy the following conditions:\\
(1) $U$ viewed as a $V$-module has the decomposition
\begin{align*}
U=M^0\oplus M^1\oplus \cdots\oplus  M^k,
\end{align*}
 such that $M^0=V, M^1,\cdots, M^k$ are nonisomorphic irreducible $V$-modules. Moreover, \begin{align*}
 \{M^0, M^1,\cdots, M^k\}=\{M^0\circ g, M^1\circ g,\cdots, M^k\circ g\},
  \end{align*} where $M^i\circ {g}$ denotes the $V$-module such that  $M^i\circ {g}=M^i$ as vector space and the vertex operator $Y_{M^i\circ {g}}(v, z)=Y_{M^i}(g(v), z)$.\\
 (2) For any two strongly regular VOA structures $(U, Y_1(\cdot, z))$ and $(U, Y_2(\cdot, z))$ on $U=M^0\oplus M^1\oplus \cdots\oplus  M^k$, there exists  a VOA isomorphism $\Psi$ from $(U, Y_1(\cdot, z))$ to  $(U, Y_2(\cdot, z))$ such that $\Psi|_V=\id$. \\
Then there exists an automorphism $\tilde g$ of $U$ such that $\tilde g(V)=V$ and that $\tilde g|_V=g$.
\end{theorem}
\pf The idea of the proof is similar to that in Theorem 2.1 of \cite{S1}. Let $\Psi:\{0,1, \cdots, k\}\to \{0,1, \cdots, k\}$ be the permutation such that
\begin{align*}
M^i\circ g\cong M^{\Psi(i)},\ \ i\in \{0,1, \cdots, k\}.
\end{align*}
Fix  $V$-module isomorphisms $\varphi_i: M^i\circ g\cong M^{\Psi(i)}$, $i\in\{0,1, \cdots, k\}$, such that $\varphi_0=\id$. Let $\psi_i: M^i\to M^i\circ g$, $i\in\{0,1, \cdots, k\}$, be the canonical linear maps such that
\begin{align*}
\psi_i(Y_{M^i}(g(v), z)w)=Y_{M^i\circ g}(v, z)\psi_i(w),
\end{align*}
for any $v\in V$ and $w\in M^i$ (cf. \cite{DM4}). In particular, we take $\psi_0=g^{-1}$. We then define a linear isomorphism $\Phi: U\to U$ by
\begin{align*}
\Phi|_{M^i}=\varphi_i\circ \psi_i,
\end{align*} and a linear map $\tilde{Y}_U(\cdot, z)$ by
\begin{align*}
\tilde Y_U(\cdot, z): U&\to (\End U)[[z^{-1},z]]\\
u&\mapsto \Phi^{-1}Y_U(\Phi(u),z)\Phi,
\end{align*}
where $u\in U$ and $Y_U(\cdot, z)$ denotes the vertex operator map of $U$. It is easy to verify that $(U, \tilde Y_U(\cdot, z))$ is also a vertex operator algebra. Moreover, we have $\Phi|_V=g^{-1}$; it implies $\tilde Y_U(u,z)v=Y_U(u, z)v$ for any $u, v\in V$. Hence, $(U, \tilde Y_U(\cdot, z))$ is also an extension vertex operator algebra of $V$. We next prove that $(U, \tilde Y_U(\cdot, z))$ is simple. Otherwise, assume that $I$ is a proper ideal of $(U, \tilde Y_U(\cdot, z))$, it is clear that $\Phi(I)$ is a proper ideal of $(U, Y_U(\cdot, z))$, this is a contradiction. By Theorem \ref{rational}, we know that $(U, \tilde Y_U(\cdot, z))$ is also a strongly regular vertex operator algebra. Then by assumption (2), there exists a linear isomorphism $\Psi: U\to U$ such that $\Psi|_V=\id$ and that
\begin{align*}
\Psi(\tilde Y_U(u^1,z)u^2)=Y_U(\Psi(u^1), z)\Psi(u^2),
\end{align*}
for any $u^1, u^2\in U$. In particular, we have
\begin{align*}
\Psi(\Phi^{-1}Y_U(\Phi(u^1),z)\Phi(u^2))=Y_U(\Psi(u^1), z)\Psi(u^2),
\end{align*}
for any $u^1, u^2\in U$. This implies that $\Psi\circ \Phi$ is an automorphism of $U$ such that $\Psi\circ \Phi|_V=g^{-1}$. Therefore, $(\Psi\circ \Phi)^{-1}$ is the desired automorphism. The proof is complete.
\qed
\vskip0.25cm
To prove that the automorphism $\theta$ of $L_{sl_9}(3, 0)$ can be lifted to an automorphism of $\widetilde{ L_{sl_9}(3, 0)}$, we need the following:
\begin{lemma}\label{dual2}
Let $L_{sl_9}(3, L(\lddot))$ be an irreducible module $L_{sl_9}(3, 0)$. Then we have
\begin{align*}
L_{sl_9}(3, L(\lddot))\circ \theta\cong L_{sl_9}(3, L(\lddot)^*)\cong L_{sl_9}(3, L(\lddot))',
\end{align*} where $L(\lddot)^*$ denotes the dual module of $L(\lddot)$.
\end{lemma}
\pf For any $x\in sl_9$, let $Y_{L_{sl_9}(3, L(\lddot))\circ \theta}(x, z)=\sum_{n\in \Z}x^\theta(n)z^{-n-1}$. By the definition of $L_{sl_9}(3, L(\lddot))\circ \theta$, we have
\begin{align*}
x^{\theta}(n)v=\theta(x)(n)v
\end{align*}
for any $x\in sl_9$, $v\in L_{sl_9}(3, L(\lddot))$ and $n\in \Z$. In particular, we have
\begin{equation}\label{dual}
\tilde h^\theta(n)v=-\tilde h(n)v,\ \
E_\alpha^\theta(n)v=-E_{-\alpha}(n)v,
\end{equation}
 for any $h\in A_8\otimes_{\Z}\C$ and $\alpha\in (A_8)_2$.

 For any irreducible $L_{sl_9}(3, 0)$-module $M$, set \begin{align*}
 \Omega(M)=\{v\in M\mid  x(n)v=0 \ {\rm for }\  x\in sl_9, n\geq 1\}.
 \end{align*}
 Then we know that $\Omega(M)$ is an irreducible $sl_9$-module (see \cite{FZ}).  Since $L_{sl_9}(3, L(\lddot))$ is irreducible, we know that $L_{sl_9}(3, L(\lddot))\circ \theta$ is also an irreducible $L_{sl_9}(3, 0)$-module. In particular, $\Omega(L_{sl_9}(3, L(\lddot))\circ \theta)$ is an irreducible $sl_9$-module. By formula (\ref{dual}), we know that $\Omega(L_{sl_9}(3, L(\lddot))\circ\theta)$ is a lowest weight $sl_9$-module with lowest weight $-\lddot$. Hence, we have $\Omega(L_{sl_9}(3, L(\lddot))\circ\theta)\cong L(\lddot)^*$. In particular, $L_{sl_9}(3, L(\lddot))\circ\theta\cong L_{sl_9}(3, L(\lddot)^*)$.

 On the other hand, by the definition of $L_{sl_9}(3, L(\lddot))'$, we have
 \begin{equation}\label{dual1}
 \la x(n)f,v \ra=\la f, -x(-n)v\ra,
 \end{equation}
 for any $x\in sl_9$, $f\in L_{sl_9}(3, L(\lddot))'$ and $n\in \Z$. Since $L_{sl_9}(3, L(\lddot))'$ is an irreducible $L_{sl_9}(3, 0)$-module, we know that $\Omega(L_{sl_9}(3, L(\lddot))')$  viewed as a vector space is equal to $L(\lddot)^*$ and is an irreducible $sl_9$-module. It follows from the formula (\ref{dual1}) that $\Omega(L_{sl_9}(3, L(\lddot))')$ viewed as an irreducible $sl_9$-module is equal to $L(\lddot)^*$. Hence, we have $L_{sl_9}(3, L(\lddot)^*)\cong L_{sl_9}(3, L(\lddot))'$. The proof is complete.\qed

 Combining Theorems \ref{unique1}, \ref{lift1} and Lemma \ref{dual2}, we have:
 \begin{theorem}
 There exists an automorphism $\tilde \theta$ of $\widetilde{L_{sl_9}(3, 0)}$ such that $\tilde \theta(L_{sl_9}(3, 0))=L_{sl_9}(3, 0)$ and that $\tilde \theta|_{L_{sl_9}(3, 0)}=\theta$.
 \end{theorem}
 \pf By Theorems \ref{unique1}, \ref{lift1}, it is sufficient to verify that
 {\tiny
 \begin{align}\label{J1}
 &\{L_{sl_9}(3, 0), L_{sl_9}(3, 3\lddot_3), L_{sl_9}(3, 3\lddot_6), L_{sl_9}(3, \lddot_1+\lddot_2+\lddot_6), L_{sl_9}(3, \lddot_3+\lddot_7+\lddot_8), L_{sl_9}(3, \lddot_4+\lddot_5)\}\notag\\
 &=\{L_{sl_9}(3, 0)\circ\theta, L_{sl_9}(3, 3\lddot_3)\circ\theta, L_{sl_9}(3, 3\lddot_6)\circ\theta, L_{sl_9}(3, \lddot_1+\lddot_2+\lddot_6)\circ\theta,\notag\\
 &\ \ \ \ \ L_{sl_9}(3, \lddot_3+\lddot_7+\lddot_8)\circ\theta, L_{sl_9}(3, \lddot_4+\lddot_5)\circ\theta\}.
 \end{align}}
 Note that $\widetilde{L_{sl_9}(3, 0)}$  is self-dual,  the formula (\ref{J1}) follows immediately from Lemma \ref{dual2}. The proof is complete.\qed
\section{Holomorphic VOA of central charge 24 with Lie algebra $A_{8,3}A_{2,1}^2$}
\def\theequation{3.\arabic{equation}}
\setcounter{equation}{0}
In this section, we shall study the holomorphic vertex operator algebra of central charge 24 with Lie algebra $A_{8,3}A_{2,1}^2$. Recall from \cite{LS} that there exists a holomorphic vertex operator algebra  $U$ such that the central charge of $U$  is $24$ and the Lie algebra $U_1$ is isomorphic to $A_{8,3}A_{2,1}^2$. Moreover, it was also proved in \cite{LS} that $U$ is strongly regular. We shall study the structure of $U$ in this section.
\subsection{Vertex operator subalgebras of $U$}\label{subsec1}
In this subsection, we shall study some vertex operator subalgebras of $U$. Note that $U$ contains a vertex operator subalgebra isomorphic to $L_{sl_9}(3,0)\ot L_{sl_3}(1,0)\ot L_{sl_3}(1,0)$. Set  $V^2=C_{U}(L_{sl_3}(1,0)\ot L_{sl_3}(1,0))$. Our goal in this subsection is to determine the structure of the vertex operator subalgebra $V^2$. 

\begin{lemma}\label{regular}
The vertex operator algebra $V^2$ is strongly regular.
\end{lemma}
\pf Since $U$ is of CFT-type, it follows that $V^2$ is of CFT-type. Note also that $V^2$ contains the vertex operator subalgebra $L_{sl_9}(3,0)$, hence $V^2$ is an extension vertex operator algebra of $L_{sl_9}(3,0)$. By Theorem \ref{rational}, to prove that $V^2$ is strongly regular, it is sufficient to show that $V^2$ is simple. Note that $L_{sl_3}(1,0)$ is isomorphic to the lattice vertex operator algebra $V_{A_2}$, where $A_2$ denotes the root lattice of type $A_2$. Let $G$ be the dual group of $(A_2\oplus A_2)^*/(A_2\oplus A_2)$. Then we know that there is an action of $G$ on $U$ such that $U^G=V^2\ot L_{sl_3}(1,0)\ot L_{sl_3}(1,0)$. It follows that $V^2\ot L_{sl_3}(1,0)\ot L_{sl_3}(1,0)$ is simple (see \cite{DM3}). This implies that $V^2$ is simple. The proof is complete. \qed
\vskip.25cm
We next determine the global dimension of $V^2$. The following result  was proved in \cite{KM}.
\begin{theorem}\label{dual3}
Let $V$ be a strongly regular vertex operator algebra and $W$ a strongly regular vertex operator subalgebra of $V$. Suppose also that the commutant $C_V(W)$ of $W$ is strongly regular and satisfies $C_{V}(C_V(W))=W$. Then all the irreducible $W$-modules appear in some simple $V$-modules.
\end{theorem}

To apply Theorem \ref{dual3}, we also need to determine the commutant $C_U(V^2)$ of $V^2$. Let $\ldot_1,\ldot_2$ be the fundamental weights of $sl_3$. It is well-known that $L_{sl_3}(1,0)$ has three nonisomorphic irreducible modules  $L_{sl_3}(1,0)$, $L_{sl_3}(1,\ldot_1)$, and $L_{sl_3}(1,\ldot_2)$. The conformal weights of $L_{sl_3}(1,0)$, $L_{sl_3}(1,\ldot_1)$, $L_{sl_3}(1,\ldot_2)$ are equal to $0$, $1/3$, $1/3$, respectively. Moreover, $L_{sl_3}(1,\ldot_1)$, $L_{sl_3}(1,\ldot_2)$ are simple current $L_{sl_3}(1,0)$-modules such that $L_{sl_3}(1,\ldot_1)\times L_{sl_3}(1,\ldot_1)\cong L_{sl_3}(1,\ldot_2)$ (see \cite{DLM1}, \cite{L3}). In particular, the fusion ring of $L_{sl_3}(1,0)$ is isomorphic to $\Z_3$.
 \begin{lemma}\label{dual4}
 The  commutant $C_U(V^2)$ of $V^2$ is equal to $L_{sl_3}(1,0)\ot L_{sl_3}(1,0)$.
\end{lemma}
\pf Note that $L_{sl_3}(1,0)\ot L_{sl_3}(1,0)$ is a vertex operator subalgebra of $C_U(V^2)$. Hence, $C_U(V^2)$ is an extension vertex operator algebra of $L_{sl_3}(1,0)\ot L_{sl_3}(1,0)$. Comparing the conformal weights of irreducible $L_{sl_3}(1,0)\ot L_{sl_3}(1,0)$-modules, we know that  $C_U(V^2)$ is equal to $L_{sl_3}(1,0)\ot L_{sl_3}(1,0)$.\qed
\vskip0.25cm
 Hence, by Theorems \ref{qdim1}, \ref{mirror}, \ref{dual3} and Lemmas \ref{regular}, \ref{dual4}, we have the following:
\begin{theorem}
(1) All the irreducible $V^2$-modules appear in  $U$. Moreover, there are $9$ nonisomorphic irreducible $V^2$-modules.\\
(2) All the irreducible $V^2$-modules are simple current modules. In particular, the fusion ring of $V^2$ is isomorphic to $\Z_3\oplus \Z_3$ and ${\rm Glob}\, V^2=9$.
\end{theorem}

To determine the vertex operator algebra structure of $V^2$, note first that $V^2$ gives rise to a modular invariant of $L_{sl_9}(3,0)$ by Theorem \ref{invariant}. On the other hand, the modular invariants of the affine vertex operator algebra $L_{sl_9}(3,0)$ have been classified in \cite{G1}. To describe the result, we need to  recall some notations from \cite{G1}. Recall that irreducible $L_{sl_9}(3,0)$-modules are parameterized by $C_{9,3}$. For any $a=(a_0,a_1,\cdots, a_8)\in C_{9,3}$, we use $Z_a$ to denote the trace function associated to the irreducible $L_{sl_9}(3,0)$-module $L_{sl_9}(3,a_1\lddot_1+\cdots+a_8\lddot_8)$ and $\la Z_{a}\ra$ to denote $\sum_{j=1}^3Z_{\rho_9^{3j}(a)}$, where $\rho_9: C_{9,3}\to C_{9,3}$ is the map defined in subsection \ref{nota1}. Let $C: C_{9,3}\to C_{9,3}$ be the map defined by $C((a_0,a_1,a_2,\cdots, a_7, a_8))=(a_0,a_8,a_7,\cdots, a_2,a_1)$. We then define a matrix $\cc$ by $\cc_{a,b}=\delta_{b, C(a)}$. For a modular invariant $\cm$ of the affine vertex operator algebra $L_{sl_9}(3,0)$, we define the {\em conjugate modular invariant} of $\cm$ to be the matrix $\cc\cm$ (cf. \cite{G1}).
\begin{theorem}\label{modular}
Any modular invariant of the affine vertex operator algebra $L_{sl_9}(3,0)$ is equal to one of the following modular invariants or their conjugate modular invariants:
\begin{center}\tiny
\begin{tabular}{|c | c|}
\hline
 $\mathcal{A}_{a,b}=\delta_{a,b}$ & $\mathcal A$ \\
\hline
$\mathcal{D}^{(1)}_{a,b}=\sum_{j=1}^9\delta^9(|d_{9,3}(a)|+6j)\delta_{b,\rho^j_9(a)}$ & $\mathcal D^{(1)}$ \\
\hline
$\mathcal{D}^{(3)}_{a,b}=\sum_{j=1}^3\delta^3(|d_{9,3}(a)|+18j)\delta_{b,\rho^{3j}_9(a)}$ & $\mathcal D^{(3)}$ \\
\hline
$\mathcal{D}^{(9)}_{a,b}=\delta^1(|d_{9,3}(a)|+54)\delta_{b,\rho^9_9(a)}$ & $\mathcal D^{(9)}$ \\
\hline
$\sum_{i=0}^2(|\la Z_{\rho^i_9(2,0,0,1,0,0,0,0,0)}\ra|^2+|\la Z_{\rho^i_9(3,0,0,0,0,0,0,0,0)}\ra|^2+|\la Z_{\rho^i_9(2,0,0,0,0,0,1,0,0)}\ra|^2$ & $\mathcal E$ \\
 $+|\la Z_{\rho^i_9(1,1,0,0,0,1,0,0,0)}\ra|^2+|\la Z_{\rho^i_9(1,0,1,0,1,0,0,0,0)}\ra|^2+2| Z_{\rho^i_9(1,0,0,1,0,0,1,0,0)}|^2$& \\
$+\la Z_{\rho^i_9(0,0,1,1,1,0,0,0,0)}\ra\overline{Z_{\rho^i_9(1,0,0,1,0,0,1,0,0)}}+Z_{\rho^i_9(1,0,0,1,0,0,1,0,0)}\overline{\la Z_{\rho^i_9(0,0,1,1,1,0,0,0,0)}\ra})$&\\
\hline
$
\sum_{i=0}^2(|\la Z_{\rho^i_9(1,0,0,0,1,1,0,0,0)}\ra+\la Z_{\rho^i_9(3,0,0,0,0,0,0,0,0)}\ra|^2+2|\la Z_{\rho^i_9(1,0,1,0,0,0,0,1,0)}\ra|^2)$ & $\mathcal E'$ \\
\hline
$|\la Z_{(1,0,0,0,1,1,0,0,0)}\ra+\la Z_{(3,0,0,0,0,0,0,0,0)}\ra|^2+|\la Z_{(1,0,1,0,0,0,0,1,0)}\ra+\la Z_{(1,0,1,0,1,0,0,0,0)}\ra|^2$ & $\mathcal E''$ \\
+$\la Z_{(0,1,0,1,0,1,0,0,0)}\ra\overline{(\sum_{i=1}^2\la Z_{\rho^i_9(3,0,0,0,0,0,0,0,0)}\ra +\la Z_{\rho^i_9(1,0,0,0,1,1,0,0,0)}\ra)}$&\\
$+(\sum_{i=1}^2\la Z_{\rho^i_9(3,0,0,0,0,0,0,0,0)}\ra +\la Z_{\rho^i_9(1,0,0,0,1,1,0,0,0)}\ra )\overline{\la Z_{(0,1,0,1,0,1,0,0,0)}\ra}$&\\
\hline
\end{tabular}
\end{center}
where $a, b\in C_{9,3}$, $\delta^x(y)=1$ if $y/x\in \Z$ and $\delta^x(y)=0$ if $y/x\notin \Z$. Here, the modular invariants of types $\ce, \ce',\ce''$ should be interpreted similarly as the modular invariants in Subsection \ref{nota1}.
\end{theorem}

To determine the modular invariant corresponding to $V^2$, we need the following:
\begin{lemma}\label{modular1}
The modular invariant corresponding to $V^2$ is not equal to the modular invariant of type $\ce''$,  $\cc\ce'$ or $\cc\ce''$.
\end{lemma}
\pf By the definition of conjugate modular invariant, the modular invariant of type  $\cc\ce'$ is equal to
{\tiny
\begin{align*}
&|\la Z_{(3,0,0,0,0,0,0,0,0)}\ra+\la Z_{(1,0,0,0,1,1,0,0,0)}\ra|^2+|\la Z_{(1,0,1,0,0,0,0,1,0)}\ra|^2\\
&+\left(\la Z_{(0,0,3,0,0,0,0,0,0)}\ra+\la Z_{(0,0,1,0,0,0,1,1,0)}\ra\right)\overline{\left(\la Z_{(0,3,0,0,0,0,0,0,0)}\ra+\la Z_{(0,1,0,0,0,1,1,0,0)}\ra\right)}\\
&+\left(\la Z_{(0,3,0,0,0,0,0,0,0)}\ra+\la Z_{(0,1,0,0,0,1,1,0,0)}\ra\right)\overline{\left(\la Z_{(0,0,3,0,0,0,0,0,0)}\ra+\la Z_{(0,0,1,0,0,0,1,1,0)}\ra\right)}\\
&+\la Z_{(1,0,1,0,1,0,0,0,0)}\ra\overline{\la Z_{(0,1,0,1,0,0,0,0,1)}\ra}+\la Z_{(0,1,0,1,0,0,0,0,1)}\ra\overline{\la Z_{(1,0,1,0,1,0,0,0,0)}\ra}.
\end{align*}}
Note that the modular invariant of type $\cc\ce'$ contains  $Z_{(0,3,0,0,0,0,0,0,0)}\overline{Z_{(0,0,3,0,0,0,0,0,0)}}$, but does not contain the term $Z_{(0,3,0,0,0,0,0,0,0)}\overline{Z_{(0,3,0,0,0,0,0,0,0)}}$, it follows from Lemma \ref{cre} that the modular invariant of type $\cc\ce'$ cannot be realized by extension vertex operator algebra.

Note also that the map $C: C_{9,3}\to C_{9,3}$ maps $\{\rho_9^i((1,0,1,0,1,0,0,0,0))|0\leq i\leq 8\}$ to itself, and that
\begin{align*}\{\rho_9^i((1,0,1,0,1,0,0,0,0))|0\leq i\leq 8\}\bigcap\{\rho_9^i((1,0,0,0,1,1,0,0,0))|0\leq i\leq 8\}=\emptyset.
\end{align*}
It follows that the modular invariants of types $\ce''$ and $\cc\ce''$ contain $$Z_{\rho_9^j(0,1,0,1,0,1,0,0,0)}\overline{Z_{(0,1,0,0,0,1,1,0,0)}}$$ for some $j$, but does not contain the term  $Z_{(0,1,0,0,0,1,1,0,0)}\overline{Z_{(0,1,0,0,0,1,1,0,0)}}$. By Lemma \ref{cre}, the modular invariants of types $\ce''$ and $\cc\ce''$ cannot be realized by extension vertex operator algebras. The proof is complete.
\qed

Combining Lemmas \ref{regular}, \ref{modular1} and Theorem \ref{modular}, we have:
\begin{theorem}
The vertex operator algebra $V^2$ is isomorphic to $\widetilde{L_{sl_9}(3, 0)}$.
\end{theorem}
\pf The idea is to show that the modular invariant associated to $V^2$ is equal to the modular invariant of type $\ce'$. Note first that the modular invariant associated to $V^2$  cannot be equal to the modular invariants of types $\ca$, $\cd^{(1)}$, $\cd^{(3)}$, $\cd^{(9)}$, $\ce$ or their conjugate modular invariants. Otherwise, $V^2$ must be a simple current extension of $L_{sl_9}(3, 0)$. However,  there are only three simple current $L_{sl_{9}}(3, 0)$-modules $L_{sl_{9}}(3, 0)$, $L_{sl_{9}}(3, 3\lddot_3)$, $L_{sl_{9}}(3, 3\lddot_6)$ that have integral conformal weights. This implies the global dimension of $L_{sl_{9}}(3, 0)$ must be less than $81$ by Theorem \ref{qdim2}. Since $L_{sl_{9}}(3, 0)$ has more than $81$ nonisomorphic irreducible modules,  this is a contradiction by Theorem \ref{qdim1}. By Lemma \ref{modular1}, the modular invariant associated to $V^2$ must be equal to the modular invariant of type $\ce'$. It follows immediately that $V^2$ viewed as a module of $L_{sl_{9}}(3, 0)$ is isomorphic to
\begin{align*}
&L_{sl_9}(3, 0)\oplus L_{sl_9}(3, 3\lddot_3)\oplus L_{sl_9}(3, 3\lddot_6)\oplus L_{sl_9}(3, \lddot_1+\lddot_2+\lddot_6)\\
&\ \ \oplus L_{sl_9}(3, \lddot_3+\lddot_7+\lddot_8)\oplus L_{sl_9}(3, \lddot_4+\lddot_5).
\end{align*}
By Theorem \ref{unique1}, we know that $V^2$ is isomorphic to $\widetilde{L_{sl_9}(3, 0)}$. The proof is complete.\qed

Since the modular invariant associated to $V^2$ is equal to the modular invariant of type $\ce'$, we can obtain the following results immediately.
\begin{theorem}
(1) There exists a unique $\widetilde{L_{sl_9}(3, 0)}$-module structure on each of the following $L_{sl_9}(3, 0)$-modules
\[
\begin{split}
\widetilde{L_{sl_9}(3, 0)}=& L_{sl_9}(3, 0)\oplus L_{sl_9}(3, 3\lddot_3)\oplus L_{sl_9}(3, 3\lddot_6)\oplus L_{sl_9}(3, \lddot_1+\lddot_2+\lddot_6)\\
&\oplus L_{sl_9}(3, \lddot_3+\lddot_7+\lddot_8)\oplus L_{sl_9}(3, \lddot_4+\lddot_5),\\
\widetilde{L_{sl_9}(3, 3\lddot_1)}= &L_{sl_9}(3, 3\lddot_1)\oplus L_{sl_9}(3, 3\lddot_4)\oplus L_{sl_9}(3, 3\lddot_7)\oplus L_{sl_9}(3, \lddot_1+\lddot_5+\lddot_6)\\ &\oplus L_{sl_9}(3, \lddot_2+\lddot_3+\lddot_7)\oplus L_{sl_9}(3, \lddot_4+\lddot_8),\\
\widetilde{L_{sl_9}(3, 3\lddot_2)}=& L_{sl_9}(3, 3\lddot_2)\oplus L_{sl_9}(3, 3\lddot_5)\oplus L_{sl_9}(3, 3\lddot_8)\oplus L_{sl_9}(3, \lddot_2+\lddot_6+\lddot_7)\\ & \oplus L_{sl_9}(3, \lddot_3+\lddot_4+\lddot_8)\oplus L_{sl_9}(3, \lddot_1+\lddot_5).
\end{split}
\]
(2) There exist two nonisomorphic  $\widetilde{L_{sl_9}(3, 0)}$-module structures on each of the following $L_{sl_9}(3, 0)$-modules
\[
\begin{split}
&L_{sl_9}(3, \lddot_2+\lddot_7)\oplus L_{sl_9}(3, \lddot_1+\lddot_3+\lddot_5)\oplus L_{sl_9}(3, \lddot_4+\lddot_6+\lddot_8),\\
&L_{sl_9}(3, \lddot_1+\lddot_3+\lddot_8)\oplus L_{sl_9}(3, \lddot_2+\lddot_4+\lddot_6)\oplus L_{sl_9}(3, \lddot_5+\lddot_7),\\
& L_{sl_9}(3, \lddot_2+\lddot_4)\oplus L_{sl_9}(3, \lddot_3+\lddot_5+\lddot_7)\oplus L_{sl_9}(3, \lddot_1+\lddot_6+\lddot_8).
\end{split}
\]
\end{theorem}

\subsection{Fusion ring of mirror extension $\widetilde{L_{sl_9}(3, 0)}$} In Subsection \ref{subsec1},  we already knew that the fusion ring of $\widetilde{L_{sl_9}(3, 0)}$ is isomorphic to $\Z_3\oplus \Z_3$. In this subsection, we shall determine the fusion ring of $\widetilde{L_{sl_9}(3, 0)}$ explicitly. First, we need to recall some facts about the Weyl group of $sl_{n+1}$. Consider the $(n+1)$-dimensional euclidean space $\R^{n+1}$. Let $\epsilon_1, \cdots, \epsilon_{n+1}$ be the standard basis of $\R^{n+1}$. It is well-known that
\begin{align*}
\{\epsilon_i-\epsilon_j|1\leq i\neq j\leq n+1\}
\end{align*}
forms a root system of type $A_n$ and that the fundamental weights are given by
\begin{align*}
\lddot_i= \frac{1}{n+1}((n+1-i)(\epsilon_1+\cdots+\epsilon_i)-i(\epsilon_i+\cdots+\epsilon_{n+1})),
\end{align*}
$1\leq i\leq n$ (see \cite{H1}). It is also known that the Weyl group of the root system of type $A_n$ is isomorphic to the permutation group $S_{n+1}$. In particular, the reflection associated to the root $\epsilon_i-\epsilon_j$ corresponds to the permutation $(i, j)$ (see \cite{H1}). The longest element of the Weyl group of root system of type $A_n$ was also determined in \cite{B1}, \cite{H1}. In particular, we have:
\begin{lemma}\label{long}
The longest element $w_0$ of the Weyl group of root system of type $A_8$ is equal to the permutation
\begin{align*}
&1\ \ 2\ \ 3\ \ 4 \ \ 5\ \ 6\ \ 7\ \ 8\ \ 9\\
&9\ \ 8\ \ 7\ \ 6 \ \ 5\ \ 4\ \ 3\ \ 2\ \ 1.
\end{align*}
In particular, we have $w_0(\lddot_i)=-\lddot_{9-i}$ for $1\leq i\leq 8$.
\end{lemma}
 We now let $L(\lddot)$ be the irreducible highest weight module of $sl_9$ with highest weight $\lddot$. It is well-known that the dual module $L(\lddot)^*$ of $L(\lddot)$ is isomorphic to $L(-w_0(\lddot))$ (see \cite{B1}, \cite{H}). Hence, we have the following:
 \begin{lemma}\label{dual5}
 Let $L(\lddot)$ be the irreducible highest weight module of $sl_9$ with highest weight $\lddot$. Then we have $L_{sl_9}(3, L(\lddot))'\cong L_{sl_9}(3, L(-w_0(\lddot)))$.
 \end{lemma}

 We now turn back to determine the fusion ring of $\widetilde{L_{sl_9}(3, 0)}$. Recall that there exist two nonisomorphic  $\widetilde{L_{sl_9}(3, 0)}$-module structures on $$L_{sl_9}(3, \lddot_2+\lddot_7)\oplus L_{sl_9}(3, \lddot_1+\lddot_3+\lddot_5)\oplus L_{sl_9}(3, \lddot_4+\lddot_6+\lddot_8).$$
We denote them by $\tau_1$ and $\tau_2$, respectively. Then we have:
 \begin{theorem}\label{fusion}
 Let $\tau_1$, $\tau_2$, $\widetilde{L_{sl_9}(3, 3\lddot_1)}$, $\widetilde{L_{sl_9}(3, 3\lddot_2)}$ be the irreducible $\widetilde{L_{sl_9}(3, 0)}$-modules defined as before. Then we have
 \begin{align*}
 &\tau_1\times \tau_1\cong \tau_2,\\
 &\tau_1\times \tau_2\cong \widetilde{L_{sl_9}(3, 0)},\\
 &\widetilde{L_{sl_9}(3, 3\lddot_1)}\times \widetilde{L_{sl_9}(3, 3\lddot_1)}\cong \widetilde{L_{sl_9}(3, 3\lddot_2)},\\
 &\widetilde{L_{sl_9}(3, 3\lddot_1)}\times \widetilde{L_{sl_9}(3, 3\lddot_2)}\cong \widetilde{L_{sl_9}(3, 0)}.
 \end{align*}
 Moreover, $\widetilde{L_{sl_9}(3, 3\lddot_1)}\times \tau_1$ and $\widetilde{L_{sl_9}(3, 3\lddot_1)}\times \tau_2$ viewed as $L_{sl_9}(3, 0)$-modules are isomorphic and $\widetilde{L_{sl_9}(3, 3\lddot_2)}\times \tau_1$ and $\widetilde{L_{sl_9}(3, 3\lddot_2)}\times \tau_2$ viewed as $L_{sl_9}(3, 0)$-modules are isomorphic.
 \end{theorem}
 \pf Note that $L_{sl_9}(3, 3\lddot_1)$ and $L_{sl_9}(3, 3\lddot_2)$ are simple current $L_{sl_9}(3, 0)$-modules. Moreover, we have the following result, which was proved in Corollary 2.27 of \cite{L3},
 \begin{align*}
 &L_{sl_9}(3, 3\lddot_1)\times L_{sl_9}(3, 3\lddot_1)\cong L_{sl_9}(3, 3\lddot_2),\ \
 &L_{sl_9}(3, 3\lddot_1)\times L_{sl_9}(3, 3\lddot_2)\cong L_{sl_9}(3, 0).
 \end{align*}
 It follows immediately that
 \begin{align*}
 &\widetilde{L_{sl_9}(3, 3\lddot_1)}\times \widetilde{L_{sl_9}(3, 3\lddot_1)}\cong \widetilde{L_{sl_9}(3, 3\lddot_2)},\ \
 &\widetilde{L_{sl_9}(3, 3\lddot_1)}\times \widetilde{L_{sl_9}(3, 3\lddot_2)}\cong \widetilde{L_{sl_9}(3, 0)}.
 \end{align*}
 We next show that $\tau_1\times \tau_1\cong \tau_2$. Recall that the fusion ring of $\widetilde{L_{sl_9}(3, 0)}$ is isomorphic to $\Z_3\oplus \Z_3$, it is good enough to show that $\tau_1'\cong \tau_2$. Note that $\tau_1$ viewed as an $L_{sl_9}(3, 0)$-module is isomorphic to $$L_{sl_9}(3, \lddot_2+\lddot_7)\oplus L_{sl_9}(3, \lddot_1+\lddot_3+\lddot_5)\oplus L_{sl_9}(3, \lddot_4+\lddot_6+\lddot_8). $$
It follows immediately from Lemmas \ref{long}, \ref{dual5} that $\tau_1'\cong \tau_1$ or $\tau_2$, which forces $\tau_1'\cong \tau_2$. Hence, we have \begin{align*}
 &\tau_1\times \tau_1\cong \tau_2,\\
 &\tau_1\times \tau_2\cong \widetilde{L_{sl_9}(3, 0)}.
 \end{align*}
 The last statement follows  from the fact that $L_{sl_9}(3, 3\lddot_1)$ and $L_{sl_9}(3, 3\lddot_2)$ are simple current $L_{sl_9}(3, 0)$-modules. The proof is complete.\qed
\subsection{Uniqueness of holomorphic VOA of central charge 24 with Lie algebra $A_{8,3}A_{2,1}^2$}
 In this subsection, we shall show that if $\tilde U$ is a holomorphic vertex operator algebra such that the central charge of $\tilde U$  is $24$ and the Lie algebra $\tilde U_1$ is isomorphic to $A_{8,3}A_{2,1}^2$, then $\tilde U$ is isomorphic to $U$. As an application, we shall construct some special automorphism of $U$.

By the similar discussion as above, we know that $C_{\tilde U}(L_{sl_3}(1,0)\ot L_{sl_3}(1,0))$ is also isomorphic to $\widetilde{L_{sl_9}(3, 0)}$. Hence, $\tilde U$ and $U$ are extension vertex operator algebras of $\widetilde{L_{sl_9}(3, 0)}\ot L_{sl_3}(1,0)\ot L_{sl_3}(1,0)$. To determine the decompositions of $\tilde U$ and $U$ viewed as $\widetilde{L_{sl_9}(3, 0)}\ot L_{sl_3}(1,0)\ot L_{sl_3}(1,0)$-modules, we need the following:
\begin{lemma}\label{weight}
Modulo integers, the conformal weights of irreducible $\widetilde{L_{sl_9}(3, 0)}$-modules are as follows:
{\tiny
$$\begin{tabular}{|c c c c c c c c|}
\hline
 $\widetilde{L_{sl_9}(3, 3\lddot_1)}$ & $\widetilde{L_{sl_9}(3, 3\lddot_2)}$  &$\tau_1$&  $ \tau_2$& $\widetilde{L_{sl_9}(3, 3\lddot_1)}\times \tau_1$& $\widetilde{L_{sl_9}(3, 3\lddot_1)}\times \tau_2$&$ \widetilde{L_{sl_9}(3, 3\lddot_2)}\times \tau_1$& $\widetilde{L_{sl_9}(3, 3\lddot_2)}\times \tau_2$\\
 \hline
 $\frac{4}{3}$&$\frac{7}{3}$&$\frac{4}{3}$&$\frac{4}{3}$&$\frac{5}{3}$&$\frac{5}{3}$&$\frac{5}{3}$&$\frac{5}{3}$\\
 \hline
\end{tabular}$$}
\end{lemma}
The following result follows immediately from Theorems \ref{mirror}, \ref{fusion} and Lemma \ref{weight}.
\begin{lemma}\label{decomposition1}
The VOA $\tilde U$ viewed as an $\widetilde{L_{sl_9}(3, 0)}\ot L_{sl_3}(1,0)\ot L_{sl_3}(1,0)$-module is isomorphic to one of the following
{\tiny\begin{align*}
&\widetilde{L_{sl_9}(3, 0)}\ot L_{sl_3}(1,0)\ot L_{sl_3}(1,0)\oplus \widetilde{L_{sl_9}(3, 3\lddot_1)}\ot L_{sl_3}(1,\ldot_1)\ot L_{sl_3}(1,\ldot_2)\oplus \widetilde{L_{sl_9}(3, 3\lddot_2)}\ot L_{sl_3}(1,\ldot_2)\ot L_{sl_3}(1,\ldot_1)\\
&\oplus \tau_1\ot L_{sl_3}(1,\ldot_1)\ot L_{sl_3}(1,\ldot_1)\oplus \tau_2\ot L_{sl_3}(1,\ldot_2)\ot L_{sl_3}(1,\ldot_2)\oplus (\widetilde{L_{sl_9}(3, 3\lddot_1)}\times \tau_1)\ot L_{sl_3}(1,\ldot_2)\ot L_{sl_3}(1,0)\\
&\oplus (\widetilde{L_{sl_9}(3, 3\lddot_2)}\times \tau_1)\ot L_{sl_3}(1,0)\ot L_{sl_3}(1,\ldot_2)\oplus (\widetilde{L_{sl_9}(3, 3\lddot_1)}\times \tau_2)\ot L_{sl_3}(1,0)\ot L_{sl_3}(1,\ldot_1)\\
&\oplus (\widetilde{L_{sl_9}(3, 3\lddot_2)}\times \tau_2)\ot L_{sl_3}(1,\ldot_1)\ot L_{sl_3}(1,0)=W^1,
\end{align*}}
{\tiny\begin{align*}
&\widetilde{L_{sl_9}(3, 0)}\ot L_{sl_3}(1,0)\ot L_{sl_3}(1,0)\oplus \widetilde{L_{sl_9}(3, 3\lddot_1)}\ot L_{sl_3}(1,\ldot_1)\ot L_{sl_3}(1,\ldot_1)\oplus \widetilde{L_{sl_9}(3, 3\lddot_2)}\ot L_{sl_3}(1,\ldot_2)\ot L_{sl_3}(1,\ldot_2)\\
&\oplus \tau_1\ot L_{sl_3}(1,\ldot_1)\ot L_{sl_3}(1,\ldot_2)\oplus \tau_2\ot L_{sl_3}(1,\ldot_2)\ot L_{sl_3}(1,\ldot_1)\oplus (\widetilde{L_{sl_9}(3, 3\lddot_1)}\times \tau_1)\ot L_{sl_3}(1,\ldot_2)\ot L_{sl_3}(1,0)\\
&\oplus (\widetilde{L_{sl_9}(3, 3\lddot_2)}\times \tau_1)\ot L_{sl_3}(1,0)\ot L_{sl_3}(1,\ldot_1)\oplus (\widetilde{L_{sl_9}(3, 3\lddot_1)}\times \tau_2)\ot L_{sl_3}(1,0)\ot L_{sl_3}(1,\ldot_2)\\
&\oplus (\widetilde{L_{sl_9}(3, 3\lddot_2)}\times \tau_2)\ot L_{sl_3}(1,\ldot_1)\ot L_{sl_3}(1,0)=W^2,
\end{align*}}
{\tiny\begin{align*}
&\widetilde{L_{sl_9}(3, 0)}\ot L_{sl_3}(1,0)\ot L_{sl_3}(1,0)\oplus \widetilde{L_{sl_9}(3, 3\lddot_1)}\ot L_{sl_3}(1,\ldot_1)\ot L_{sl_3}(1,\ldot_1)\oplus \widetilde{L_{sl_9}(3, 3\lddot_2)}\ot L_{sl_3}(1,\ldot_2)\ot L_{sl_3}(1,\ldot_2)\\
&\oplus \tau_1\ot L_{sl_3}(1,\ldot_2)\ot L_{sl_3}(1,\ldot_1)\oplus \tau_2\ot L_{sl_3}(1,\ldot_1)\ot L_{sl_3}(1,\ldot_2)\oplus (\widetilde{L_{sl_9}(3, 3\lddot_1)}\times \tau_1)\ot L_{sl_3}(1,0)\ot L_{sl_3}(1,\ldot_2)\\
&\oplus (\widetilde{L_{sl_9}(3, 3\lddot_2)}\times \tau_1)\ot L_{sl_3}(1,\ldot_1)\ot L_{sl_3}(1,0)\oplus (\widetilde{L_{sl_9}(3, 3\lddot_1)}\times \tau_2)\ot L_{sl_3}(1,\ldot_2)\ot L_{sl_3}(1,0)\\
&\oplus (\widetilde{L_{sl_9}(3, 3\lddot_2)}\times \tau_2)\ot L_{sl_3}(1,0)\ot L_{sl_3}(1,\ldot_1)=W^3,
\end{align*}}
{\tiny\begin{align*}
&\widetilde{L_{sl_9}(3, 0)}\ot L_{sl_3}(1,0)\ot L_{sl_3}(1,0)\oplus \widetilde{L_{sl_9}(3, 3\lddot_1)}\ot L_{sl_3}(1,\ldot_2)\ot L_{sl_3}(1,\ldot_2)\oplus \widetilde{L_{sl_9}(3, 3\lddot_2)}\ot L_{sl_3}(1,\ldot_1)\ot L_{sl_3}(1,\ldot_1)\\
&\oplus \tau_1\ot L_{sl_3}(1,\ldot_1)\ot L_{sl_3}(1,\ldot_2)\oplus \tau_2\ot L_{sl_3}(1,\ldot_2)\ot L_{sl_3}(1,\ldot_1)\oplus (\widetilde{L_{sl_9}(3, 3\lddot_1)}\times \tau_1)\ot L_{sl_3}(1,0)\ot L_{sl_3}(1,\ldot_1)\\
&\oplus (\widetilde{L_{sl_9}(3, 3\lddot_2)}\times \tau_1)\ot L_{sl_3}(1,\ldot_2)\ot L_{sl_3}(1,0)\oplus (\widetilde{L_{sl_9}(3, 3\lddot_1)}\times \tau_2)\ot L_{sl_3}(1,\ldot_1)\ot L_{sl_3}(1,0)\\
&\oplus (\widetilde{L_{sl_9}(3, 3\lddot_2)}\times \tau_2)\ot L_{sl_3}(1,0)\ot L_{sl_3}(1,\ldot_2)=W^4,
\end{align*}}
{\tiny\begin{align*}
&\widetilde{L_{sl_9}(3, 0)}\ot L_{sl_3}(1,0)\ot L_{sl_3}(1,0)\oplus \widetilde{L_{sl_9}(3, 3\lddot_1)}\ot L_{sl_3}(1,\ldot_1)\ot L_{sl_3}(1,\ldot_2)\oplus \widetilde{L_{sl_9}(3, 3\lddot_2)}\ot L_{sl_3}(1,\ldot_2)\ot L_{sl_3}(1,\ldot_1)\\
&\oplus \tau_1\ot L_{sl_3}(1,\ldot_2)\ot L_{sl_3}(1,\ldot_2)\oplus \tau_2\ot L_{sl_3}(1,\ldot_1)\ot L_{sl_3}(1,\ldot_1)\oplus (\widetilde{L_{sl_9}(3, 3\lddot_1)}\times \tau_1)\ot L_{sl_3}(1,0)\ot L_{sl_3}(1,\ldot_1)\\
&\oplus (\widetilde{L_{sl_9}(3, 3\lddot_2)}\times \tau_1)\ot L_{sl_3}(1,\ldot_1)\ot L_{sl_3}(1,0)\oplus (\widetilde{L_{sl_9}(3, 3\lddot_1)}\times \tau_2)\ot L_{sl_3}(1,\ldot_2)\ot L_{sl_3}(1,0)\\
&\oplus (\widetilde{L_{sl_9}(3, 3\lddot_2)}\times \tau_2)\ot L_{sl_3}(1,0)\ot L_{sl_3}(1,\ldot_2)=W^5,
\end{align*}}
{\tiny\begin{align*}
&\widetilde{L_{sl_9}(3, 0)}\ot L_{sl_3}(1,0)\ot L_{sl_3}(1,0)\oplus \widetilde{L_{sl_9}(3, 3\lddot_1)}\ot L_{sl_3}(1,\ldot_2)\ot L_{sl_3}(1,\ldot_2)\oplus \widetilde{L_{sl_9}(3, 3\lddot_2)}\ot L_{sl_3}(1,\ldot_1)\ot L_{sl_3}(1,\ldot_1)\\
&\oplus \tau_1\ot L_{sl_3}(1,\ldot_2)\ot L_{sl_3}(1,\ldot_1)\oplus \tau_2\ot L_{sl_3}(1,\ldot_1)\ot L_{sl_3}(1,\ldot_2)\oplus (\widetilde{L_{sl_9}(3, 3\lddot_1)}\times \tau_1)\ot L_{sl_3}(1,\ldot_1)\ot L_{sl_3}(1,0)\\
&\oplus (\widetilde{L_{sl_9}(3, 3\lddot_2)}\times \tau_1)\ot L_{sl_3}(1,0)\ot L_{sl_3}(1,\ldot_2)\oplus (\widetilde{L_{sl_9}(3, 3\lddot_1)}\times \tau_2)\ot L_{sl_3}(1,0)\ot L_{sl_3}(1,\ldot_1)\\
&\oplus (\widetilde{L_{sl_9}(3, 3\lddot_2)}\times \tau_2)\ot L_{sl_3}(1,\ldot_2)\ot L_{sl_3}(1,0)=W^6,
\end{align*}}
{\tiny\begin{align*}
&\widetilde{L_{sl_9}(3, 0)}\ot L_{sl_3}(1,0)\ot L_{sl_3}(1,0)\oplus \widetilde{L_{sl_9}(3, 3\lddot_1)}\ot L_{sl_3}(1,\ldot_2)\ot L_{sl_3}(1,\ldot_1)\oplus \widetilde{L_{sl_9}(3, 3\lddot_2)}\ot L_{sl_3}(1,\ldot_1)\ot L_{sl_3}(1,\ldot_2)\\
&\oplus \tau_1\ot L_{sl_3}(1,\ldot_2)\ot L_{sl_3}(1,\ldot_2)\oplus \tau_2\ot L_{sl_3}(1,\ldot_1)\ot L_{sl_3}(1,\ldot_1)\oplus (\widetilde{L_{sl_9}(3, 3\lddot_1)}\times \tau_1)\ot L_{sl_3}(1,\ldot_1)\ot L_{sl_3}(1,0)\\
&\oplus (\widetilde{L_{sl_9}(3, 3\lddot_2)}\times \tau_1)\ot L_{sl_3}(1,0)\ot L_{sl_3}(1,\ldot_1)\oplus (\widetilde{L_{sl_9}(3, 3\lddot_1)}\times \tau_2)\ot L_{sl_3}(1,0)\ot L_{sl_3}(1,\ldot_2)\\
&\oplus (\widetilde{L_{sl_9}(3, 3\lddot_2)}\times \tau_2)\ot L_{sl_3}(1,\ldot_2)\ot L_{sl_3}(1,0)=W^7,
\end{align*}}
{\tiny\begin{align*}
&\widetilde{L_{sl_9}(3, 0)}\ot L_{sl_3}(1,0)\ot L_{sl_3}(1,0)\oplus \widetilde{L_{sl_9}(3, 3\lddot_1)}\ot L_{sl_3}(1,\ldot_2)\ot L_{sl_3}(1,\ldot_1)\oplus \widetilde{L_{sl_9}(3, 3\lddot_2)}\ot L_{sl_3}(1,\ldot_1)\ot L_{sl_3}(1,\ldot_2)\\
&\oplus \tau_1\ot L_{sl_3}(1,\ldot_1)\ot L_{sl_3}(1,\ldot_1)\oplus \tau_2\ot L_{sl_3}(1,\ldot_2)\ot L_{sl_3}(1,\ldot_2)\oplus (\widetilde{L_{sl_9}(3, 3\lddot_1)}\times \tau_1)\ot L_{sl_3}(1,0)\ot L_{sl_3}(1,\ldot_2)\\
&\oplus (\widetilde{L_{sl_9}(3, 3\lddot_2)}\times \tau_1)\ot L_{sl_3}(1,\ldot_2)\ot L_{sl_3}(1,0)\oplus (\widetilde{L_{sl_9}(3, 3\lddot_1)}\times \tau_2)\ot L_{sl_3}(1,\ldot_1)\ot L_{sl_3}(1,0)\\
&\oplus (\widetilde{L_{sl_9}(3, 3\lddot_2)}\times \tau_2)\ot L_{sl_3}(1,0)\ot L_{sl_3}(1,\ldot_1)=W^8.
\end{align*}}
\end{lemma}
To prove the uniqueness of holomorphic VOA of central charge 24 with Lie algebra $A_{8,3}A_{2,1}^2$, we also need the following:
\begin{proposition}\label{auto}
Let $V$ be a strongly regular vertex operator algebra and $g$ an automorphism of $V$. Let $ V^1,V^2$ be extension vertex operator algebras of $V$. Assume that $V, g, V^1,V^2$ satisfy the following conditions:\\
(1) There exists a unique vertex operator algebra structure on $V^2$ such that $V^2$ is an extension vertex operator algebra of $V$.\\
(2) $V^1\circ g$ viewed as a $V$-module  is isomorphic to $V^2$. Moreover, there exists a $V$-module isomorphism $\phi: V^1\circ g\to V^2$ such that $\phi|_V=g^{-1}$.\\
Then the vertex operator algebra $V^1$ is isomorphic to $V^2$.
\end{proposition}
\pf Let $Y_1$ and $Y_2$ be the vertex operator maps of $V^1$ and $V^2$, respectively. By assumption, there exists a $V$-module isomorphism $\phi$ from  $V^1\circ g$ to $V^2$  such that $\phi|_V=g^{-1}$.  In particular, for any $u, v\in V$, we have
\begin{align*}
\phi(Y_1(\phi^{-1}(u),z)v)=Y_2(u,z)\phi(v).
\end{align*}Define a linear map
\begin{align*}
Y^g_2: V^2&\to \End(V^2)[[z^{-1}, z]]\\
v&\mapsto \phi Y_1(\phi^{-1}(v), z)\phi^{-1}.
\end{align*}
It is easy to show that $(V^2, Y^g_2)$ is a vertex operator algebra.
Note also that $Y_2^g(u, z)v=Y_2(u,z)v$ for any $u,v\in V$. Hence, $(V^2, Y^g_2)$ is also an extension vertex operator algebra of $V$. By assumption, there exists a linear map $\psi: V^2\to V^2$ such that $\psi(Y_2^g(w^1, z)w^2)=Y_2(\psi(w^1),z)\psi(w^2)$ for any $w^1, w^2\in V^2$. This implies $$\psi\phi(Y_1(v^1, z)v^2)=Y_2(\psi\phi(v^1),z)\psi\phi(v^2)$$ for any $v^1, v^2\in V^1$. Hence, the vertex operator algebra $V^1$ is isomorphic to $V^2$. The proof is complete.
\qed
\vskip0.25cm
We are now ready to show the uniqueness of holomorphic VOA of central charge 24 with Lie algebra $A_{8,3}A_{2,1}^2$. The idea is to apply Proposition \ref{auto}. We first recall some automorphisms of $ L_{sl_3}(1,0)\ot L_{sl_3}(1,0)$. Let $\varphi$ be the diagram automorphism of $sl_3$. Then we know that $\varphi$ induces an automorphism of $L_{sl_3}(1,0)$, which is also denoted by $\varphi$. By a direct calculation, we can show that $L_{sl_3}(1,\ldot_1)\circ\varphi\cong L_{sl_3}(1,\ldot_2)$ and $L_{sl_3}(1,\ldot_2)\circ\varphi\cong L_{sl_3}(1,\ldot_1)$. We also need an automorphism $\sigma$ of $L_{sl_3}(1,0)\ot L_{sl_3}(1,0)$, which is defined by
\begin{align*}
\sigma(v\ot w)= w\ot v,
\end{align*}
for any $ v, w\in L_{sl_3}(1,0)$.  By a direct calculation, for any $L_{sl_3}(1,0)$-modules $M^1$ and $M^2$, we have $(M^1\ot M^2)\circ\sigma\cong M^2\ot M^1$.
\begin{theorem}\label{main1}
Let $\tilde U^1$, $\tilde U^2$ be holomorphic vertex operator algebras such that the central charges of $\tilde U^1$, $\tilde U^2$  are equal to $24$ and the Lie algebras $\tilde U^1_1$, $\tilde U^2_1$ are isomorphic to $A_{8,3}A_{2,1}^2$. Then there exists a vertex operator algebra isomorphism $\Phi:\tilde U^1\to \tilde U^2$.
\end{theorem}
\pf We shall show that there exists a vertex operator algebra isomorphism from $\tilde U^1$ to $\tilde U^2$. Note that $\tilde U^1$ and $\tilde U^2$ are simple current extensions of $\widetilde{L_{sl_9}(3, 0)}\ot L_{sl_3}(1,0)\ot L_{sl_3}(1,0)$, it follows from  Proposition 5.3 of \cite{DM1} that $\tilde U^1$ and $\tilde U^2$ viewed as extension vertex operator algebras of $\widetilde{L_{sl_9}(3, 0)}\ot L_{sl_3}(1,0)\ot L_{sl_3}(1,0)$ are unique. In particular, if $\tilde U^1$ and $\tilde U^2$ viewed as $\widetilde{L_{sl_9}(3, 0)}\ot L_{sl_3}(1,0)\ot L_{sl_3}(1,0)$-modules have the same decomposition, then $\tilde U^1$ is isomorphic to $\tilde U^2$. We then assume that $\tilde U^1$, $\tilde U^2$ viewed as $\widetilde{L_{sl_9}(3, 0)}\ot L_{sl_3}(1,0)\ot L_{sl_3}(1,0)$-modules have different decompositions. Let $G$ be the subgroup of $\Aut (\widetilde{L_{sl_9}(3, 0)}\ot L_{sl_3}(1,0)\ot L_{sl_3}(1,0))$ generated by $\{1\ot \varphi\ot 1, 1\ot 1\ot \varphi, 1\ot \varphi\ot \varphi, 1\ot \sigma\}$. By Lemma \ref{decomposition1}, we can prove that there exists an automorphism $g\in G$ such that $\tilde U^1\circ g$ viewed as an $\widetilde{L_{sl_9}(3, 0)}\ot L_{sl_3}(1,0)\ot L_{sl_3}(1,0)$-module is isomorphic to $\tilde U^2$. By Proposition \ref{auto}, we know that the vertex operator algebra $\tilde U^1$ is isomorphic to $\tilde U^2$. The proof is complete.\qed

Furthermore, we have the following:
\begin{theorem}\label{main}
Let $(\tilde U^1, Y_1(\cdot, z))$, $(\tilde U^2, Y_2(\cdot, z))$ be holomorphic vertex operator algebras such that the central charges of $\tilde U^1$, $\tilde U^2$  are equal to $24$ and the Lie algebras $\tilde U^1_1$, $\tilde U^2_1$ are isomorphic to $A_{8,3}A_{2,1}^2$. Assume further that $\tilde U^1$, $\tilde U^2$ viewed as $L_{sl_9}(3, 0)\ot L_{sl_3}(1,0)\ot L_{sl_3}(1,0)$-modules have the same decomposition, then there exists a vertex operator algebra isomorphism $\Phi:\tilde U^1\to \tilde U^2$ such that  $\Phi|_{L_{sl_9}(3, 0)\ot L_{sl_3}(1,0)\ot L_{sl_3}(1,0)}=\id$.
\end{theorem}
\pf First, note that both $\tilde U^1$ and  $\tilde U^2$ have a vertex operator algebra isomorphic to $\widetilde{L_{sl_9}(3, 0)}\ot L_{sl_3}(1,0)\ot L_{sl_3}(1,0)$. By Theorem 3.6, we know that there exists a vertex operator algebra isomorphism
\begin{align*}
f:\widetilde{L_{sl_9}(3, 0)}\ot L_{sl_3}(1,0)\ot L_{sl_3}(1,0)\to \widetilde{L_{sl_9}(3, 0)}\ot L_{sl_3}(1,0)\ot L_{sl_3}(1,0)
 \end{align*} such that $f|_{L_{sl_9}(3, 0)\ot L_{sl_3}(1,0)\ot L_{sl_3}(1,0)}=\id$ and that $f(Y_1(u, z)v)=Y_2(f(u), z)f(v)$ for any $u, v\in \widetilde{L_{sl_9}(3, 0)}\ot L_{sl_3}(1,0)\ot L_{sl_3}(1,0)$. Consider the linear isomorphism $\tilde f: \tilde U^1\to \tilde U^2$ such that
 \begin{align*}
 &\tilde f|_{\widetilde{L_{sl_9}(3, 0)}\ot L_{sl_3}(1,0)\ot L_{sl_3}(1,0)}=f,\\
 &\tilde f|_{(\widetilde{L_{sl_9}(3, 0)}\ot L_{sl_3}(1,0)\ot L_{sl_3}(1,0))^{\perp}}=\id,
 \end{align*}
 where $(\widetilde{L_{sl_9}(3, 0)}\ot L_{sl_3}(1,0)\ot L_{sl_3}(1,0))^{\perp}$ denotes the complement $\widetilde{L_{sl_9}(3, 0)}\ot L_{sl_3}(1,0)\ot L_{sl_3}(1,0)$-module of $\widetilde{L_{sl_9}(3, 0)}\ot L_{sl_3}(1,0)\ot L_{sl_3}(1,0)$ in $\tilde U^1$. Define a new vertex operator by
 \begin{align*}
 Y_3(\cdot,z): &\tilde U^1\to \End(\tilde U^1)[[z^{-1}, z]],\\
 &u\mapsto \tilde f^{-1}Y_2(\tilde f(u),z)\tilde f,
 \end{align*} for any $u\in \tilde U^1$. It is easy to show that $(\tilde U^1, Y_3(\cdot,z))$ is a vertex operator algebra. Moreover, $(\tilde U^1, Y_3(\cdot,z))$ is also an extension vertex operator algebra of $\widetilde{L_{sl_9}(3, 0)}\ot L_{sl_3}(1,0)\ot L_{sl_3}(1,0)$. Since $(\tilde U^1, Y_1(\cdot,z))$ and $(\tilde U^1, Y_3(\cdot,z))$ are simple current extensions of $\widetilde{L_{sl_9}(3, 0)}\ot L_{sl_3}(1,0)\ot L_{sl_3}(1,0)$, we know that there exists a linear isomorphism $g:\tilde U^1\to \tilde U^1$ such that  $g|_{\widetilde{L_{sl_9}(3, 0)}\ot L_{sl_3}(1,0)\ot L_{sl_3}(1,0)}=\id$ and that $g(Y_1(u, z)v)=Y_3(g(u), z)g(v)$ for any $u, v\in \tilde U^1$ (see \cite{DM1}). As a result, $\tilde f\circ g$ is the desired isomorphism. The proof is complete.
\qed

\subsection{Liftings of $\theta\ot \sigma$}\label{sec:4.4}
In this subsection, we shall show that there exists an automorphism $\widetilde{\theta\ot \sigma}$ of $U$ such that $\widetilde{\theta\ot \sigma}|_{L_{sl_9}(3, 0)\ot L_{sl_3}(1,0)\ot L_{sl_3}(1,0)}=\theta\ot \sigma$  and  has order two.  Note that by Theorem \ref{main} we may assume that $U$ viewed as an $\widetilde{L_{sl_9}(3, 0)}\ot L_{sl_3}(1,0)\ot L_{sl_3}(1,0)$-module is isomorphic to $W^1$. By Lemmas \ref{dual2}, \ref{long}, \ref{dual5} and Theorems \ref{lift1}, \ref{main}, we immediately have:
\begin{theorem}\label{lift}
There exists an automorphism $\overline{\theta\ot \sigma}$ of $U$ such that  \begin{align*}
\overline{\theta\ot \sigma}(L_{sl_9}(3, 0)\ot L_{sl_3}(1,0)\ot L_{sl_3}(1,0))=L_{sl_9}(3, 0)\ot L_{sl_3}(1,0)\ot L_{sl_3}(1,0)
 \end{align*}and $\overline{\theta\ot \sigma}|_{L_{sl_9}(3, 0)\ot L_{sl_3}(1,0)\ot L_{sl_3}(1,0)}=\theta\ot \sigma$.
\end{theorem}

Note that the order of $\overline{\theta\ot \sigma}$ may not be two; we next show that there exists a lifting $\widetilde{\theta\ot \sigma}$ of $\theta\ot \sigma$ which has order two.  First,  we need to study automorphisms of $U$ such that the restrictions on $L_{sl_9}(3,0)\otimes L_{sl_3}(1,0)\otimes L_{sl_3}(1,0)$ are trivial.

\begin{lemma}\label{identity1}
Let $f$ be an automorphism of $\widetilde{L_{sl_9}(3, 0)}\ot L_{sl_3}(1,0)\ot L_{sl_3}(1,0)$ such that $$f({L_{sl_9}(3, 0)}\ot L_{sl_3}(1,0)\ot L_{sl_3}(1,0))={L_{sl_9}(3, 0)}\ot L_{sl_3}(1,0)\ot L_{sl_3}(1,0)$$ and $f|_{{L_{sl_9}(3, 0)}\ot L_{sl_3}(1,0)\ot L_{sl_3}(1,0)}=\id$. Then $f=\id$.
\end{lemma}
\pf Note that for any irreducible $L_{sl_9}(3,0)\otimes L_{sl_3}(1,0)\otimes L_{sl_3}(1,0)$-submodule $M$ of $\widetilde{L_{sl_9}(3, 0)}\ot L_{sl_3}(1,0)\ot L_{sl_3}(1,0)$, we have $f(M)=M$. Recall that  $\widetilde{L_{sl_9}(3, 0)}$ has the decomposition
\begin{align*}
\widetilde{L_{sl_9}(3, 0)}=&L_{sl_9}(3, 0)\oplus L_{sl_9}(3, 3\lddot_3)\oplus L_{sl_9}(3, 3\lddot_6)\oplus L_{sl_9}(3, \lddot_1+\lddot_2+\lddot_6)\\
&\ \ \oplus L_{sl_9}(3, \lddot_3+\lddot_7+\lddot_8)\oplus L_{sl_9}(3, \lddot_4+\lddot_5),
\end{align*}
and that $L_{sl_9}(3, 3\lddot_3)$ and $L_{sl_9}(3, 3\lddot_6)$ are simple current modules of $L_{sl_9}(3, 0)$ such that $L_{sl_9}(3, 3\lddot_3)\times L_{sl_9}(3, 3\lddot_3)=L_{sl_9}(3, 3\lddot_6)$ and $L_{sl_9}(3, 3\lddot_3)\times L_{sl_9}(3, 3\lddot_6)=L_{sl_9}(3, 0)$. By the Schur's Lemma, we know that $$f|_{L_{sl_9}(3, 3\lddot_3)\otimes L_{sl_3}(1,0)\otimes L_{sl_3}(1,0)}=\lambda\id,~~~f|_{L_{sl_9}(3, 3\lddot_6)\otimes L_{sl_3}(1,0)\otimes L_{sl_3}(1,0)}=\lambda^2\id,$$ where $\lambda$ denotes a cube root of unity. Furthermore, by fusion rules between $L_{sl_9}(3, 0)$-modules, we know that there exists a complex number $\nu$ such that $$f|_{L_{sl_9}(3, \lddot_4+\lddot_5)\otimes L_{sl_3}(1,0)\otimes L_{sl_3}(1,0)}=\nu\id,$$ and that
$$f|_{L_{sl_9}(3, \lddot_1+\lddot_2+\lddot_6)\otimes L_{sl_3}(1,0)\otimes L_{sl_3}(1,0)}=\nu\lambda\id,~~~f|_{L_{sl_9}(3, \lddot_3+\lddot_7+\lddot_8)\otimes L_{sl_3}(1,0)\otimes L_{sl_3}(1,0)}=\nu\lambda^2\id.$$
We next prove that $\lambda=1$. Otherwise, $\lambda$ is a primitive cube root of unity. Since $\widetilde{L_{sl_9}(3, 0)}\ot L_{sl_3}(1,0)\ot L_{sl_3}(1,0)$ is a strongly regular VOA, then there exists a nondegenerate  symmetric bilinear form $\la,\ra$ on $\widetilde{L_{sl_9}(3, 0)}\ot L_{sl_3}(1,0)\ot L_{sl_3}(1,0)$. As a consequence, there exist elements $u\in L_{sl_9}(3, \lddot_1+\lddot_2+\lddot_6)\ot L_{sl_3}(1,0)\ot L_{sl_3}(1,0)$, $v\in L_{sl_9}(3, \lddot_3+\lddot_7+\lddot_8)\ot L_{sl_3}(1,0)\ot L_{sl_3}(1,0)$ and $w\in (L_{sl_9}(3, 0)\oplus L_{sl_9}(3, 3\lddot_3)\oplus L_{sl_9}(3, 3\lddot_6))\ot L_{sl_3}(1,0)\ot L_{sl_3}(1,0)$ such that $$\la Y_{\widetilde{L_{sl_9}(3, 0)}\ot L_{sl_3}(1,0)\ot L_{sl_3}(1,0)}(u,z)v, w\ra\neq 0.$$ This implies that $\nu^2$ is equal to $1$, $\lambda$ or $\lambda^2$. Hence, $\nu$ is equal to $1$, $\lambda$, $\lambda^2$, $-1$, $-\lambda$, or $-\lambda^2$.
Suppose first that $\nu$ is equal to $1$. Then the fixed point subVOA of $f$ is equal to $L_{sl_9}(3, 0)\oplus L_{sl_9}(3, \lddot_4+\lddot_5)$, which is not a simple current extension of $L_{sl_9}(3, 0)$. By Theorem 3.15 of \cite{Lin}, there exists an extension VOA $L_{sl_3}(9, 0)\oplus L_{sl_3}(9, 4\dot{\Lambda}_1+4\dot{\Lambda}_2)$ of $L_{sl_3}(9, 0)$. This contradicts to Theorem  3.10 of \cite{AL}. By the similar argument, $\nu$ cannot be  $\lambda$ or $\lambda^2$. Suppose that $\nu$ is equal to $-1$, $-\lambda$, or $-\lambda^2$. Then $f^3$ acts on $(L_{sl_9}(3, 0)\oplus L_{sl_9}(3, 3\lddot_3)\oplus L_{sl_9}(3, 3\lddot_6))\ot L_{sl_3}(1,0)\ot L_{sl_3}(1,0)$ as $\id$ and on $(L_{sl_9}(3, \lddot_1+\lddot_2+\lddot_6)\oplus L_{sl_9}(3, \lddot_3+\lddot_7+\lddot_8)\oplus L_{sl_9}(3, \lddot_4+\lddot_5))\ot L_{sl_3}(1,0)\ot L_{sl_3}(1,0)$ as $-\id$. Therefore, $L_{sl_9}(3, \lddot_1+\lddot_2+\lddot_6)\oplus L_{sl_9}(3, \lddot_3+\lddot_7+\lddot_8)\oplus L_{sl_9}(3, \lddot_4+\lddot_5)$ must be a simple current $L_{sl_9}(3, 0)\oplus L_{sl_9}(3, 3\lddot_3)\oplus L_{sl_9}(3, 3\lddot_6)$-module. This is a contradiction. Hence, we have $\lambda=1$. As a consequence, we have $\nu^2=1$. This further forces that $\nu=1$. The proof is complete. \qed
\begin{theorem}\label{identity}
Let $g$ be an automorphism of $U$ such that  $g|_{L_{sl_9}(3,0)\otimes L_{sl_3}(1,0)\otimes L_{sl_3}(1,0)}=\id$. Then the order of $g$ is equal to $1$ or $3$.
\end{theorem}
\pf Note that for any irreducible $L_{sl_9}(3,0)\otimes L_{sl_3}(1,0)\otimes L_{sl_3}(1,0)$-submodule $M$ of $U$, we have $g(M)=M$. In particular, we have $g(\widetilde{L_{sl_9}(3, 0)}\ot L_{sl_3}(1,0)\ot L_{sl_3}(1,0))=\widetilde{L_{sl_9}(3, 0)}\ot L_{sl_3}(1,0)\ot L_{sl_3}(1,0)$. Thus, $g$ induces an automorphism of $\widetilde{L_{sl_9}(3, 0)}\ot L_{sl_3}(1,0)\ot L_{sl_3}(1,0)$ such that $$g({L_{sl_9}(3, 0)}\ot L_{sl_3}(1,0)\ot L_{sl_3}(1,0))={L_{sl_9}(3, 0)}\ot L_{sl_3}(1,0)\ot L_{sl_3}(1,0)$$ and $g|_{{L_{sl_9}(3, 0)}\ot L_{sl_3}(1,0)\ot L_{sl_3}(1,0)}=\id$. By Lemma \ref{identity1}, we have $g|_{\widetilde{L_{sl_9}(3, 0)}\ot L_{sl_3}(1,0)\ot L_{sl_3}(1,0)}=\id$. Note that $U$ is a simple current extension of $\widetilde{L_{sl_9}(3, 0)}\ot L_{sl_3}(1,0)\ot L_{sl_3}(1,0)$, and the fusion ring of $\widetilde{L_{sl_9}(3, 0)}\ot L_{sl_3}(1,0)\ot L_{sl_3}(1,0)$ is isomorphic to $\Z_3\oplus\Z_3\oplus\Z_3\oplus\Z_3$. By Proposition 4.2.9 of \cite{Y}, we know that the order of $g$ must be equal to $1$ or $3$. The proof is complete.
\qed

As a corollary, we have:
\begin{corollary}\label{lift_inv}
There exists an order two automorphism $\widetilde {\theta\otimes \sigma}$ of $U$ such that $$\widetilde {\theta\otimes \sigma}({L_{sl_9}(3, 0)}\ot L_{sl_3}(1,0)\ot L_{sl_3}(1,0))={L_{sl_9}(3, 0)}\ot L_{sl_3}(1,0)\ot L_{sl_3}(1,0)$$ and $\widetilde {\theta\otimes \sigma}|_{{L_{sl_9}(3, 0)}\ot L_{sl_3}(1,0)\ot L_{sl_3}(1,0)}={\theta\otimes \sigma}$.
\end{corollary}
\pf Let $\overline{\theta\otimes \sigma}$ be a lift of $\theta\otimes \sigma$ as in Theorem \ref{lift}. Then $$(\overline{\theta\otimes \sigma})^2|_{{L_{sl_9}(3, 0)}\ot L_{sl_3}(1,0)\ot L_{sl_3}(1,0)}=\id.$$ By Theorem \ref{identity}, we know that the order of $\overline{\theta\otimes \sigma}$ is equal to $2k$, where $k=1$ or $3$. Then $(\overline{\theta\otimes \sigma})^k$ is the desired automorphism. \qed



\section{$\widetilde{\theta\ot \sigma}$-twisted module}
Let $U$ be a strongly regular holomorphic VOA of central charge $24$ and $U_1\cong A_{8,3}A_{2,1}^2$. By Theorem \ref{main1}, the VOA structure of $U$ is uniquely determined.  Let $g=\widetilde{\theta\ot \sigma}$ be the automorphism as given in
Corollary  \ref{lift_inv}. We shall study the unique irreducible $g$-twisted module $U^T$ of $U$ in this section. In particular, we show that the lowest (conformal) weight of $U^T$ is $1$.

Recall that $U$ contains a full subVOA $\widetilde{L_{sl_9}(3, 0)}\ot L_{sl_3}(1,0)\ot L_{sl_3}(1,0)$ and $g|_{\widetilde{L_{sl_9}(3, 0)}}=\tilde{\theta}$ and
$g|_{L_{sl_3}(1,0)\ot L_{sl_3}(1,0)}=\sigma$. Therefore, $U^T$ is a direct sum of the tensor products of irreducible $\tilde{\theta}$-twisted $\widetilde{L_{sl_9}(3, 0)}$-modules and  irreducible $\sigma$-twisted $L_{sl_3}(1,0)\ot L_{sl_3}(1,0)$-modules. Since $\widetilde{L_{sl_9}(3, 0)}$ is an extension of $L_{sl_9}(3, 0)$, irreducible  $\tilde{\theta}$-twisted modules of $\widetilde{L_{sl_9}(3, 0)}$ are direct sum of irreducible $\theta$-twisted $L_{sl_9}(3, 0)$-modules.

\subsection{$\Z_2$-twisted module of lattice VOA}\label{sec5.1}
First, let us recall the construction of irreducible $\Z_2$-twisted modules of lattice VOA from \cite{DL1} (see also \cite{FLM}).

Let $L$ be an even lattice with a symmetric bilinear form $\la \ ,\ \ra$ and let $\sigma$ be an isometry of $L$ of order $2$.
Let $\mathfrak{h}=L \otimes_{\mathbb{Z}}\mathbb{C}$ and extend the bilinear form  $\langle\ , \ \rangle$ $\C$-bilinearly to $\mathfrak{h}$ .  Set
$\mathfrak{h}_{(0)}=\{\mathfrak{h}\mid \sigma x=x\}$  and
$\mathfrak{h}_{(1)}=\{\mathfrak{h}\mid \sigma x=- x\}.$  For $i=0,1$, let $P_{i}$ be the natural projection of $\mathfrak{h}$ to $\mathfrak{h}_{(i)}$.
We also use $x_{(i)}$ to denote $P_i(x)$ for any $x\in \mathfrak{h}$.

The twisted affine algebra $\hat{\frak{h}}[\sigma]$ is the Lie algebra
\begin{equation*}
\frak{\hat{h}}[\sigma]= \frak{h}_{(0)}\otimes \C[t, t^{-1}] \oplus \frak{h}_{(1)}\otimes t^{1/2}\C[t, t^{-1}] \oplus \mathbb{C}c
\end{equation*}
with the bracket given by
\begin{equation*}
\left[ x\otimes t^{m},y\otimes t^{n}\right] =m \langle x,y\rangle \delta
_{m+n,0}c \text{\qquad and \qquad} \left[
c,\frak{\hat{h}}_{\mathbb{Z}+\frac{1}{2}}\right] =0
\end{equation*}
where $x,y\in \frak{h}_{(i)}$ and $m,n\in \mathbb{Z}+\frac{i}{2}, i=0,1$. The Lie algebra $\frak{\hat{h}}[\sigma]$ has a triangular decomposition given by
\begin{equation*}
\frak{\hat{h}}[\sigma]= \frak{\hat{h}}^+[\sigma]\oplus\frak{\hat{h}}^0[\sigma] \oplus
\frak{\hat{h}}^-[\sigma]
\end{equation*}
where $\frak{\hat{h}}^{\pm}[\sigma]=\bigoplus _{n=1}^{\infty }%
\left( \frak{h}_{(0)} \otimes t^{\pm n} \oplus \frak{h}_{(1)} \otimes t^{\pm (n-\frac{1}{2})}\right)$ and $\frak{\hat{h}}^{0}[\sigma]=\frak{h}_{(0)}\oplus \mathbb{C}c$.

Let $N=(1-P_0)\mathfrak{h} \cap L=\{\alpha\in L\mid \langle \alpha, \mathfrak{h}_{(0)}\rangle=0\}$ and $M= (1-\sigma)L$.  Note that $M<N$ since $\langle M, \mathfrak{h}_{(0)}\rangle=0$. Let $\langle \kappa \rangle$ be a cyclic group of order $2$.  On $N$, we define
\[
C_N(\alpha, \beta)= \kappa^{\langle \alpha, \beta\rangle},  \quad \text{ and } \quad
R= \{ \alpha\in N \mid C_N(\alpha, N)=1\}.
\]

Let
\[
1\longrightarrow \langle \kappa\rangle \longrightarrow\hat{N} \overset{\varphi}\longrightarrow N
\longrightarrow 1
\]
be the  central extension of $N$ associated with the commutator map $C_N$ and  let $\hat{\sigma}$ be a lift of $\sigma$ in $Aut(\hat{N})$, i.e.,  $\varphi(\hat{\sigma}(a))=\sigma(\varphi(a))$ for any $a\in \hat{N}$.  Set $K=\{a\hat{\sigma}(a)^{-1}\mid a\in \hat{N}\}$. Then $K$ is an index $2$ subgroup of $\hat{M}$.

\begin{proposition}[\cite{Le}]
For any irreducible character $\chi: \hat{R}/K \to \C$ with $\chi(\kappa K)=-1$, there is a unique irreducible $\hat{N}/K$-module $T_\chi$ such that $\hat{R}$ acts according to $\chi$. Moreover, $\dim T_\chi =|N/R|^{1/2}$.
\end{proposition}

Let $L^*=\{\alpha\in L\otimes_{\Z}\C\mid \langle \alpha, L\rangle< \Z\}$ be the dual lattice of $L$. For any coset $\lambda+P_0(L) \in P_0(L^*)/P_0(L)$ and an irreducible character $\chi\in Irr(\hat{R}/K)$, denote
\[
V_L^{T_{\chi, \lambda}} = S(\frak{\hat{h}}^-[\sigma]) \otimes \C[\lambda+P_0(L)] \otimes T_\chi.
\]
It is shown in \cite{DL1} (see also \cite{Le}) that $V_L^{T_{\chi, \lambda}}$ is an irreducible $\hat{\sigma}$-twisted module of $V_L$. Moreover, the lowest (conformal) weight of $V_L^{T_{\chi, \lambda}}$ is given by  $$\frac{\mathrm{rank}(N)}{16} + \frac{n_\lambda}2,$$
where $n_\lambda= min\{ \langle \alpha, \alpha\rangle\mid \alpha\in \lambda+P_0(L)\}$.

\subsection{$\sigma$-twisted module of $L_{sl_3}(1,0)\otimes L_{sl_3}(1,0)$}
Next we study the $\sigma$-twisted modules for $L_{sl_3}(1,0)\otimes L_{sl_3}(1,0)$. Recall that $L_{sl_3}(1,0)$ is isomorphic to the lattice VOA $V_{A_2}$ and hence, $L_{sl_3}(1,0)\otimes L_{sl_3}(1,0)\cong V_{A_2\oplus A_2}$.

Let $\sigma: A_2\oplus A_2\to A_2\oplus A_2$ be an isometry defined by $\sigma(\alpha, \beta)=(\beta, \alpha)$ and set $A^+= \{(\alpha, \alpha)\mid \alpha\in A_2\}$ and $A^-= \{(\alpha, -\alpha)\mid \alpha\in A_2\}$. Then $A^+$ and $A^-$ are the eigenlattices of $\sigma$  of eigenvalues $+1$ and $-1$, respectively.

Let $P_0=\frac{1}2(1+\sigma)$ be the natural projection from $(A_2\oplus A_2)^*$ to $\mathfrak{h}_{(0)}$. Then
\[
P_0(A_2\oplus A_2)=\frac{1}2 A^+ \quad \text{ and } \quad (1-P_0)\mathfrak{h}\cap (A_2^2)= A^-=(1-\sigma)(A_2^2).
\]
Moreover, we have $P_0((A_2^*)^2)= \frac{1}2 \{ (\alpha, \alpha)\mid \alpha\in A_2^*\}$ and $|P_0((A_2^*)^2)/ P_0(A_2^2)|=3$.

Let $P_0((A_2^*)^2)=P_0(A_2^2)\cup (\lambda_1+ P_0(A_2^2))\cup (\lambda_2+ P_0(A_2^2))$ be the coset decomposition of $P_0(A_2^2)$  in $P_0((A_2^*)^2)$. Then, by \cite{DL1} (see also \cite{BDM}), we have the following lemma.

\begin{lemma} \label{sweight}
There are $3$ inequivalent irreducible $\sigma$-twisted modules for $V_{A_2\oplus A_2}$ and they are given by
\[
\begin{split}
W_0&=S(\hat{\mathfrak{h}}^-[\sigma])\otimes \C[P_0(A_2^2)]\otimes T, \\
W_1&=S(\hat{\mathfrak{h}}^-[\sigma])\otimes \C[\lambda_1+P_0(A_2^2)]\otimes T, \text{ and} \\
W_2& =S(\hat{\mathfrak{h}}^-[\sigma])\otimes \C[\lambda_2+P_0(A_2^2)]\otimes T, \\
\end{split}
\]
where $T$ is a one-dimensional irreducible module of $\hat{A}^-$.
The lowest weight of $W_0$ is $2/16=1/8$ and the lowest weights of $W_1$ and $W_2$ are $1/8+ 1/6=7/24$.
\end{lemma}

\medskip

\subsection{Irreducible $\theta$-twisted modules of $L_{sl_9}(3,0)$}
Next we consider the irreducible $\theta$-twisted modules of $L_{sl_9}(3,0)$.

\begin{lemma}
There are exactly $5$ irreducible $\theta$-invariant modules for $L_{sl_9}(3,0)$, namely, $L_{sl_9}(3,0)$, $L_{sl_9}(3,\lddot_1+\lddot_8)$, $L_{sl_9}(3,\lddot_2+\lddot_7)$, $L_{sl_9}(3,\lddot_3+\lddot_6)$, and $L_{sl_9}(3,\lddot_4+\lddot_5)$.
\end{lemma}

\begin{proof}
The lemma follows immediately from Lemma \ref{dual2}.
\end{proof}

By \cite[Theorem 1.1]{DLM3}, we also have the following result.

\begin{corollary}\label{ttwisted}
There are exactly $5$ inequivalent irreducible $\theta$-twisted modules for $L_{sl_9}(3,0)$.
\end{corollary}

Now let $L=A_8^3$ be   a root lattice of type $A_8^3$. For explicit calculations, we use the standard model for $A_8$, i.e.,
\[
A_8=\{ (a_1, \dots, a_9)\in \Z^9\mid \sum_{i=1}^9 a_i=0\}.
\]

The following lemma can be obtained by a direct calculation.
\begin{lemma}
Every coset of $A_8/ 2A_8$ contains a vector of norm $\leq 8$ and the coset representatives (with minimal norm) are given as follows:

\begin{center}
\begin{tabular}{|cc|cc|cc|}
\hline
Norm && Representatives && \# of cosets &\\ \hline \hline
$0$   &&  $0$ &&  $1$ &\\
$2$   && $\pm (1 ,-1, 0^7)$ &&  $36$ &\\
$4$   && $\pm (1^2, -1^2, 0^5)$ && $126$&\\
$6$   && $\pm (1^3, -1^3, 0^3)$ && $84$&\\
$8$   && $\pm (1^4, -1^4, 0)$   && $8$&\\ \hline
\end{tabular}
\end{center}
\end{lemma}

\medskip
Let $X$ and $Y$  be sublattices of $A_8$ such that $X/2A_8$ and $Y/2A_8$ are maximal totally singular subspaces of $A_8/2A_8$ and $X+Y=A_8$. Recall that $A_8/2A_8$ forms a non-singular quadratic space associated with the standard quadratic form $q(\alpha+2A_8)=\langle \alpha, \alpha\rangle/2 \mod 2$ since $\det(A_8)=9$.

Set  $\mu_{i,j}= \eta_i-\eta_j$ for $i\neq j$ and $\eta=\eta_1+\eta_2 +\eta_3$ and let $\Phi$ be the sublattice  of $L=A_8^3$ spanned by
\[
\mu_{1,2}(A_8) \cup \eta(X) \cup 2L
\]
and let $\Psi$ be the sublattice spanned by
\[
\mu_{2,3}(A_8) \cup \eta(Y) \cup 2L.
\]
Then $\Phi/ 2L$ and $\Psi/2L$ are maximal totally singular subspaces of $L/2L$ and
$\Phi+\Psi =L$. Note that $\mu_{i,j}(A_8)\perp \eta(A_8)$.

Let $\chi_0$ be an irreducible character of  $\hat{\Phi}/ K$ such that $\chi_0(\iota(e_\alpha))=1$ and $\chi_0(\kappa K)=-1$. Then
\[
T=\mathrm{Ind}_{ \hat{\Phi}/ K}^{\hat{L}/K} F_{\chi_0},
\]
where $F_{\chi_0}$ is the irreducible $\hat{\Phi}/ K$-module affording the character $\chi_0$.

Take $0\neq t_0\in F_{\chi_0}$. Then $F_{\chi_0}=\C t_0$ and
\[
T=Span_\C\{ e_{\mu_{2,3}(\beta)+ \eta(\gamma)}\cdot t_0\mid \beta\in A_8, \gamma\in Y\}
\]
For simplicity, we set $t_{(\beta,\gamma)} = e_{\mu_{2,3}(\beta)+ \eta(\gamma)}\cdot t_0$.

Recall from Sec. \ref{a83inA83} that the lattice VOA $V_{A_8^3}$ contains a subVOA isomorphic to the affine VOA $L_{sl_9}(3, 0)$, which is generated by
\[
\begin{split}
\tilde h& =\eta(h)(-1)\cdot\1 \quad \text{ for } h\in A_8\ot _{\Z}\C, \\
E_\alpha& =\iota(e_{\eta_1(\alpha)})+\iota(e_{\eta_2(\alpha)})+\iota(e_{\eta_3(\alpha)})\quad  \text{ for }\alpha\in (A_8)_2.
\end{split}
\]
The conformal element $\Omega$ of $L_{sl_9}(3,0)$ is given by
\[
\begin{split}
\Omega&=\omega_E +\frac{3}4 \omega_{P} -\frac{1}{12} \sum_{\al\in (A_8)_2\atop 1\leq i<j\leq 3} e_{\mu_{i,j}(\al)},
\end{split}
\]
where  $E= \eta(A_8)$ , $P = \{(\alpha,\beta,\gamma)\in A_8^3 \mid \alpha +\beta+\gamma=0\}$ and $\omega_M$ denotes the conformal element of the lattice VOA $V_M$ .

\medskip

Next we shall construct some explicit eigenvectors of $\Omega_1$ on $T$.

\begin{lemma}\label{lem:5.3}
For $\alpha\in (A_8)_2$ and  $(\beta, \gamma)\in A_8\times Y$, we have
\[
\begin{split}
&\ (e_{\mu_{1,2}(\al)}+ e_{\mu_{2,3}(\al)}+ e_{\mu_{1,3}(\al)})_1 t_{(\beta, \gamma)}\\ = &
\begin{cases}
\frac{1}{16} t_{(\beta, \gamma)} & \text{ if } \la\alpha,\beta\ra=0 \mod 2,\\
\frac{1}{16}(2t_{(\alpha+\beta, \gamma)} - t_{(\beta, \gamma)})& \text{ if } \la \alpha,\beta\ra=1 \mod 2.\\
\end{cases}
\end{split}
\]
\end{lemma}
\begin{proof}
Recall from Sec. \ref{sec:3.3} that
\[
\begin{split}
e_{\mu_{1,2}(\al)}e_{\mu_{2,3}(\beta)} & = (-1)^{\langle \alpha, \beta\rangle}  e_{\mu_{2,3}(\beta)}e_{\mu_{1,2}(\al)};\\
e_{\mu_{2,3}(\al)}e_{\mu_{2,3}(\beta)} &= e_{\mu_{2,3}(\alpha+ \beta)};\\
e_{\mu_{1,3}(\al)}& = (-1)^{\epsilon_0(\alpha,\alpha)} e_{\mu_{2,3}(\al)}e_{\mu_{1,2}(\al)}
= - e_{\mu_{2,3}(\al)}e_{\mu_{1,2}(\al)}.
\end{split}
\]
Since  $e_{\mu_{1,2}(\al)}\cdot t_0=t_0$ for all $\alpha \in (A_8)_2$, we have
\[
\begin{split}
&\ (e_{\mu_{1,2}(\al)}+ e_{\mu_{2,3}(\al)}+ e_{\mu_{1,3}(\al)})_1 t_{(\beta, \gamma)}\\
= &\frac{1}{2^4}(e_{\mu_{1,2}(\al)}+ e_{\mu_{2,3}(\al)}+ e_{\mu_{1,3}(\al)})\cdot e_{\mu_{2,3}(\beta)} e_{\eta(\gamma)}\cdot  t_0\\
=& \frac{1}{16}\left( (-1)^{\langle \alpha, \beta\rangle}  t_{(\beta, \gamma)} + t_{(\alpha+\beta, \gamma)}
- (-1)^{\langle \alpha, \beta\rangle}  t_{(\alpha+\beta, \gamma)}\right).
\end{split}
\]
Thus we have the desired conclusion.
\end{proof}

\begin{Notation}
For $n=0,1,2,3,4$, let $\mathcal{C}_{2n}$ be the set of cosets of $2A_8$ in $A_8$ with minimal norm $2n$, i.e., $\mathcal{C}_{2n}=\{ \alpha+2A_8\mid min\{ \langle a, a\rangle| a\in \alpha+2A_8\} = 2n\}.$
Note that a coset $\alpha+2A_8\in A_8/2A_8$ is in $\mathcal{C}_{2n}$ if it contains a vector of the shape $(1^n, -1^n, 0^{9-2n})$.
\end{Notation}

\begin{lemma}
Let $\mu= (1, -1, 0^7)$ and denote $v_0=t_0$ and
\[
v_1=21t_{(\mu,0)} +\sum_{{\langle \beta, \mu\rangle =0\,mod\,2}
\atop {\beta +2A_8\in \mathcal{C}_{2}\setminus \{\mu+2A_8\}}} t_{(\beta,0)} - 3 \sum_{{\langle \beta, \mu\rangle =1\,mod\,2}
\atop {\beta +2A_8\in \mathcal{C}_{2}}} t_{(\beta,0)}.
\]
Then we have
\[
\Omega_1 v_0= \frac{7}8 v_0, \quad \text{ and } \quad \Omega_1 v_1= \frac{29}{24} v_1. \]
\end{lemma}

\begin{proof}
First we note that
\[
(\omega_E)_1 t = \frac{8}{16} t\quad \text{ and } \quad (\omega_{P})_1 t= \frac{16}{16}t
\]
for any $t\in T$. Note that $\mathrm{rank}(E)=8$ and $\mathrm{rank}(P)=16$.

For $v_0=t_0$, we have
\[
\left(\sum_{\al\in (A_8)_2\atop 1\leq i<j\leq 3} e_{\mu_{i,j}(\al)}\right)_1v_0 = \frac{72}{16} v_0
\]
by Lemma \ref{lem:5.3}. Notice that there are $72$ elements in $(A_8)_2$. Hence,
\[
\Omega_1 v_0 = \frac{1}2 v_0+ \frac{3}4 v_0-\frac{1}{12} \cdot \frac{72}{16} v_0 =\frac{7}8 v_0.
\]

For any $\beta +2A_8 \in \mathcal{C}_2$, we have
\[
\begin{split}
|\{\alpha\in (A_8)_2\mid \la \beta,\alpha\ra=0\mod 2\}|=44; \\
|\{\alpha\in (A_8)_2\mid \la \beta,\alpha\ra=1\mod 2\}|=28.
\end{split}
\]

Let $\mu=(1,-1, 0^7)$. If $\beta+2A_8=\mu+2A_8$, then $\la \alpha,\beta\ra=1\mod 2$ implies
$$\la \alpha+\beta,\mu\ra=\la \alpha+\beta, \beta\ra = 1 \mod 2$$ and there are $28$ such $\alpha\in (A_8)_2$.

 If $\beta+2A_8\neq \mu+2A_8$ and $\la \beta, \mu\ra=0\mod 2$, then
\[
\begin{split}
|\{\alpha\in (A_8)_2\mid \la \beta,\alpha\ra=1, \la\alpha+\beta,\mu\ra= 0\mod 2\}|=20; \\
|\{\alpha\in (A_8)_2\mid \la \beta,\alpha\ra=1, \la \alpha+\beta,\mu\ra= 1\mod 2\}|=8.
\end{split}
\]
If $\la \beta,\mu\ra=1\mod 2$, then
\[
\begin{split}
&|\{\alpha\in (A_8)_2\mid \la \beta,\alpha\ra=1, \la\alpha+\beta,\mu\ra= 0\mod 2, \alpha+\beta+2A_8\neq \mu+2A_8\}|=12; \\
&|\{\alpha\in (A_8)_2 \mid \alpha+\beta+2A_8= \mu+2A_8\}| =2; \\
&|\{\alpha\in (A_8)_2\mid \la \beta,\alpha\ra=1, \la \alpha+\beta,\mu\ra= 1\mod 2\}|=14.
\end{split}
\]
Then by Lemma \ref{lem:5.3}, we have
\[
\begin{split}
& \ 16 \left(\sum_{\al\in (A_8)_2\atop 1\leq i<j\leq 3} e_{\mu_{i,j}(\al)}\right)_1v_1 \\
= &\  (21\cdot (44-28) -3\cdot 2\cdot 28) t_{(\mu,0)} \\
& + ((44-28)+2\cdot 20 -3\cdot2\cdot 8) \sum_{{\langle \beta, \mu\rangle =0\,mod\,2}
\atop {\beta +2A_8\in \mathcal{C}_{2n}\setminus \{\mu+2A_8\}}} t_{(\beta,0)}\\
& - (3\cdot (44-28) +3\cdot 2\cdot 14 -2\cdot 12 -21\cdot 2 \cdot 2 ) \sum_{{\langle \beta, \mu\rangle =1\,mod\,2}
\atop {\beta +2A_8\in \mathcal{C}_{2n}}} t_{(\beta,0)}\\
= & 8\left(21t_{(\mu,0)} +\sum_{{\langle \beta, \mu\rangle =0\,mod\,2}
\atop {\beta +2A_8\in \mathcal{C}_{2n}\setminus \{\mu+2A_8\}}} t_{(\beta,0)} - 3 \sum_{{\langle \beta, \mu\rangle =1\,mod\,2}
\atop {\beta +2A_8\in \mathcal{C}_{2n}}} t_{(\beta,0)}\right) = 8 v_1.
\end{split}
\]
Hence we have
\[
\Omega_1 v_1 = \frac{1}2 v_1+ \frac{3}4 v_1-\frac{1}{12} \cdot \frac{8}{16} v_1 =\frac{29}{24} v_1.
\]
\end{proof}

\begin{Notation}
For any $\mu\in A_8$, we denote
\[
P_{2n}^\mu = \sum_{{\langle \beta, \mu\rangle =0\,mod\,2}
\atop {\beta +2A_8\in \mathcal{C}_{2n}}}  t_{(\beta,0)},\quad  \text{ and }\quad
N_{2n}^\mu = \sum_{{\langle \beta, \mu\rangle =1\,mod\,2}
\atop {\beta +2A_8\in \mathcal{C}_{2n}}}  t_{(\beta,0)}.
\]
\end{Notation}

\begin{lemma}
Let $\mu_4= (1^2, -1^2, 0^5)$, $\mu_6= (1^3, -1^3, 0^3)$ and
$\mu_8= (1^4, -1^4, 0)$ and denote
$v_2=10P_2^{\mu_4} -7N_2^{\mu_4}$, $v_3= 6P_2^{\mu_6} -7N_2^{\mu_6}$, and $v_4= 5P_8^{\mu_8} -N_8^{\mu_8}$. Then we have
\[
\begin{split}
\Omega_1 v_2= \frac{59}{48} v_2,\qquad
& \Omega_1 v_3= \frac{55}{48} v_3, \quad \text{ and }\quad  \Omega_1 v_4= \frac{17}{16} v_4.
\end{split}
\]
\end{lemma}

\begin{proof}
The proof is similar to the previous lemma. First we have
\[
(\omega_E+ \frac{3}4 \omega_{P} )_1 t =  \frac{5}4 t
\]
for any $t\in T$.

Let $\mu=\mu_4=(1^2, -1^2, 0^5)$.
For any $\beta +2A_8 \in \mathcal{C}_2$, we have
\[
\begin{split}
|\{\alpha\in (A_8)_2\mid \la\beta,\alpha\ra=0\mod 2\}|=44; \\
|\{\alpha\in (A_8)_2\mid \la\beta,\alpha\ra=1\mod 2\}|=28.
\end{split}
\]
If $\la\beta,\mu\ra=0\mod 2$, then
\[
\begin{split}
|\{\alpha\in (A_8)_2\mid \la\beta,\alpha\ra=1, \la \alpha+\beta,\mu\ra= 0\mod 2\}|=8; \\
|\{\alpha\in (A_8)_2\mid \la\beta,\alpha\ra=1, \la\alpha+\beta,\mu\ra= 1\mod 2\}|=20.
\end{split}
\]
If $\la\beta,\mu\ra=1\mod 2$, then
\[
\begin{split}
&|\{\alpha\in (A_8)_2\mid \la\beta,\alpha\ra=1, \la\alpha+\beta,\mu\ra= 0\mod 2\}|=14; \\
&|\{\alpha\in (A_8)_2\mid \la\beta,\alpha\ra=1, \la\alpha+\beta,\mu\ra= 1\mod 2\}|=14.
\end{split}
\]
Then by Lemma \ref{lem:5.3}, we have
\[
\begin{split}
& \ 16 \left(\sum_{\al\in (A_8)_2\atop 1\leq i<j\leq 3} e_{\mu_{i,j}(\al)}\right)_1v_2 \\
= & ( 10\cdot (16+16) - 40\cdot 7) P_2^{\mu_4} - ( 7\cdot(16+28) -28\cdot 10) N_2^{\mu_4}\\
=& 40 P_2^{\mu_4} - 28 N_2^{\mu_4} = 4 v_2.
\end{split}
\]
Hence we have
\[
\Omega_1 v_2 = \frac{5}4 v_1-\frac{1}{12} \cdot \frac{4}{16} v_2 =\frac{59}{48} v_2.
\]
The other cases can be proved by the similar method.
\end{proof}

By Corollary \ref{ttwisted} and the lemmas above, we have the following result.

\begin{lemma}\label{tweight}
There are $5$ inequivalent irreducible $\theta$-twisted modules for $L_{sl_9}(3,0)$ and their lowest conformal weights are  $7/8$, $29/24$, $59/48$, $55/48$ and $17/16$.
\end{lemma}

\begin{lemma}
Let $\alpha$ be a root of $A_8$ and $(\beta, \gamma)\in A_8\times Y$. Then
\[
E_\alpha\cdot t_{(\beta, \gamma)}
=
\begin{cases}
t_{(\beta, \gamma+\alpha)} & \text{ if } \langle \alpha, \beta\rangle =0\mod 2,\\
2t_{(\beta+\alpha, \gamma+\alpha)}- t_{(\beta, \gamma+\alpha)} & \text{ if } \langle \alpha, \beta\rangle =1\mod 2.
\end{cases}
\]
\end{lemma}

\begin{proof}
First we note that
\[
\begin{split}
e_{\eta_1(\alpha)}&= e_{\eta(\alpha)}e_{\mu_{2,3}(\alpha)}e_{-2\eta_2(\alpha)},\\
e_{\eta_2(\alpha)}&= - e_{\eta(\alpha)}e_{\mu_{2,3}(\alpha)} e_{-\mu_{1,2}(\alpha)}e_{-2\eta_3(\alpha)},\\
e_{\eta_3(\alpha)}&= e_{\eta(\alpha)}e_{\mu_{1,2}(\alpha)}e_{-2\eta_2(\alpha)}.
\end{split}
\]
Thus,
\[
\begin{split}
E_\alpha&=(e_{\eta_1(\alpha)}+e_{\eta_2(\alpha)}+e_{\eta_3(\alpha)})\cdot t_{\beta, \gamma}\\
& = t_{(\beta+\alpha, \gamma+\alpha)} -(-1)^{\langle \alpha, \beta\rangle}t_{(\beta+\alpha, \gamma+\alpha)} +  (-1)^{\langle \alpha, \beta\rangle}t_{(\beta, \gamma+\alpha)}
\end{split}
\]
and we have the desired result.
\end{proof}

\begin{lemma}\label{M0}
Let $M$ be the $\theta$-twisted $L_{sl_9}(3,0)$-module generated by $t_{0}$.  Let $M(0)$ be the top module of an irreducible (twisted or untwisted) module  $M$. Then we have
\[
M(0) =Span_\C\{ t_{(0, \gamma)}\mid \gamma \in Y\}.
\]
In particular, $M(0)$ has dimension $16$.
\end{lemma}

\begin{proposition}\label{twistedmodule}
Let $U$ be a holomorphic VOA of central charge $24$ and $U_1\cong A_{8,3}A_{2,1}^2$. Let $g=\widetilde{\theta\otimes \sigma}$ be an involution of $U$ as given in Theorem \ref{lift}.
Let $U^T$ be the unique irreducible $g$-twisted module of $U$. Then the lowest conformal weight of $U^T$ is $1$ and $\dim(U^T_1)\geq 16$.
\end{proposition}

\begin{proof}
We first note that $U^T$ is a direct sum of the tensor products of irreducible $\tilde{\theta}$-twisted $\widetilde{L_{sl_9}(3, 0)}$-modules and  irreducible $\sigma$-twisted $L_{sl_3}(1,0)\ot L_{sl_3}(1,0)$-modules. Moreover, every $\tilde{\theta}$-twisted $\widetilde{L_{sl_9}(3, 0)}$-module is a direct sum of irreducible $\theta$-twisted $L_{sl_9}(3, 0)$-modules.

By \cite[Theorem 1.6]{DLM3}, the conformal weights of $U^T$ are in $\frac{1}4 \Z$.  Then by
Lemmas \ref{sweight} and \ref{tweight}, the lowest conformal weights for irreducible
$\widetilde{\theta\otimes \sigma}$-twisted $\widetilde{L_{sl_9}(3, 0)}\ot L_{sl_3}(1,0)\ot L_{sl_3}(1,0)$-submodules of $U^T$ are $1 (=7/8+1/8)$ and $ 3/2(=29/24 +7/24)$.
Therefore,  the lowest conformal weight of $U^T$ is $1$. That $\dim(U^T_1)\geq 16$ follows from Lemma \ref{M0}.
\end{proof}

\section{$\widetilde{\theta\otimes \sigma}$-orbifold construction}
Let $U$, $U^T$ and $g$ be defined as in Proposition \ref{twistedmodule}.
Since the conformal weights of $U^T$ are in $\frac{1}2 \Z$, we can apply an orbifold construction to $U$ using $g$ (see \cite[Theorem 5.15]{EMS} and \cite{CM}) and obtain
a strongly regular holomorphic VOA $$\tilde{U}(g) = U^g \oplus  (U^T)_\Z$$ of central charge $24$.  The following lemma is a generalization of \cite[Theorem 4.3]{LS} (see also \cite{Mo}).

\begin{lemma}\label{dim}
Let $U$ and $g$ be as above. Then
$$
\dim U_1+\dim \tilde{U}(g)_1=3\dim(U^{g})_1+24(1-\dim (U^T)_{1/2}).
$$
\end{lemma}
 \begin{proof}
Let $Z_U(g,\tau)=q^{-c/24}\sum_{n=0}^\infty {\rm Tr\ }g|_{U_n}q^{n}$ be the trace function of $g$ on $U$ and let $Z_{U^T}(\tau)=q^{-c/24}\sum_{n=1}^\infty\dim (U^T)_n q^{n/2}$ be
the character of $U^T$, where $q=e^{2\pi \sqrt{-1} \tau}$ and $\tau$ is in the upper half plane $\mathfrak{H}$.

It was proved in \cite{DLM3} that $Z_U(g,\tau)$ and $Z_{U^T}(\tau)$ both converge to holomorphic functions in $\mathfrak{H}$ and
\[
Z_U(g,S\tau)= Z_U(g,-\frac{1}{\tau})= \lambda Z_{U^T}(\tau)
\]
for some $\lambda \in \mathbb{C}$. Moreover, it was proved in
\cite[Proposition 5.5]{EMS} that $\lambda=1$. Therefore, $U$ and $U^T$ satisfy Assumptions (A1) and (A2) of \cite[Section 4.2]{LS}. Hence the proof of Theorem 4.3 of \cite{LS} still holds for $U$ and $g$ and we have
$$
\dim U_1+\dim \tilde{U}(g)_1=3\dim(U^{g})_1+24(1-\dim (U^T)_{1/2}),
$$
as desired.
 \end{proof}

By a direct calculation, we also have the following.
\begin{lemma}
Let $U$, $U^T$ and $g$ be defined as above.

1. The weight one Lie algebra of $U^g$ has the type $B_{4}A_{2}$ and has dimension $44$.

2. $\dim(\tilde{U}(g)_1)=60$.
\end{lemma}

\begin{proof}
Recall that  $g|_{L_{sl_9}(3,0)\ot L_{sl_3}(1,0)\ot L_{sl_3}(1,0)}=\theta \otimes \sigma$.  Since the fixed point Lie algebra of $\theta$ on  $sl_9$ has type $B_4$ and the fixed point of $\sigma$ on  $sl_3\oplus sl_3$ has type $A_2$, we have (1).

For (2), we have $\dim(\tilde{U}(g)_1)= 3\times 44 +24 - 96=60$ by Lemma \ref{dim}.
\end{proof}

\begin{theorem}
Let $U$, $U^T$ and $g$ be defined as above. Then $\tilde{U}(g) = U^g \oplus  (U^T)_\Z$ is a strongly regular holomorphic VOA of central charge $24$ and $ \tilde{U}(g)_1$ has the type $F_{4,6}A_{2,2}$.
\end{theorem}

\begin{proof}
Since $\dim(\tilde{U}(g)_1)=60$, the ratio $\frac{h^\vee}{k}=\frac{60-24}{24}=3/2$ (cf. \cite{DM1}). Therefore, the dual Coxeter number of any simple ideal of $\tilde{U}(g)_1$ must be divisible by $3$ and hence a simple ideal must have the type $A_2$, $C_2$, $A_5$, $C_5$, $D_4$, or $F_4$. That $\dim(\tilde{U}(g)_1)=60$ implies that $\tilde{U}(g)_1$ has the type $C_{2,2}^6$, $D_{4,4}A_{2,2}^4$ or $F_{4,6}A_{2,2}$.
Since $\tilde{U}(g)_1$ contains a Lie subalgebra of type  $B_{4}A_{2}$, $\tilde{U}(g)_1= F_{4,6}A_{2,2}$ is the only possibility.
\end{proof}

\section{Uniqueness of holomorphic VOAs of central charge 24 with weight one Lie algebras $F_{4,6}A_{2,2}$ and $E_{7,3}A_{5,1}$}
\def\theequation{7.\arabic{equation}}
\setcounter{equation}{0}
In this section, we shall prove that holomorphic VOAs  of central charge 24 with weight one Lie algebras $F_{4,6}A_{2,2}$ and $E_{7,3}A_{5,1}$ are unique.

\subsection{Reverse orbifold construction of holomorphic VOAs}\label{sec1}First, we recall the reverse orbifold method from \cite{LSn16} (see also \cite{KLL}). Let $V$ be a strongly regular holomorphic VOA of central charge $24$, $g$ be an~automorphism of $V$ of prime order $p$.
 We then know that there is a~unique $g^r$-twisted $V$-module $V^{\mathrm{T}}(g^r)$
for each $1\leq r\leq p-1$ (\cite[Theorem~10.3]{DLM3}).
Moreover, the fixed point subspace $V^g$ of $V$ with respect to $g$
is a~sub\,VOA of $V$.
We say that the pair $(V,g)$ satisfies the {\it orbifold condition} if
 there exists a~unique simple VOA $\tilde{V}$
 such that $V^g$ is embedded in $\tilde{V}$ and
$\tilde{V}\cong V^g\oplus \bigoplus_{r=1}^{p-1}V^{\mathrm{T}}(g^r)_\mathbb{Z}$ as a~$V^g$-module,
where $V^{\mathrm{T}}(g^r)_\Z$  is the subspace of $V^{\mathrm{T}}(g^r)$ of integral conformal weights (cf.~\cite{EMS}).
If $(V,g)$ satisfies the orbifold condition,  the VOA
$\tilde{V}$ which satisfies
the above assumptions is strongly regular and holomorphic. We refer to $\tilde{V}$ as the VOA obtained by applying the $\Z_p$-orbifold
construction to $V$ and $g$,  and denote the VOA $\tilde{V}$ by $\tilde{V}(g)$.
If we further define an~automorphism $a=a_{V,g}$ of $\tilde{V}(g)$ by
$a|_{V^{g}}=1$ and $a|_{V^{\mathrm{T}}(g^r)_\Z}=e^{2\pi\sqrt{-1}r/p}$ ($1\leq r\leq p-1$), we then know that the pair $(\tilde{V}(g),a)$ satisfies the orbifold condition
and  $\widetilde{\tilde{V}(g)}(a)\cong V$ (see \cite{EMS}).

We are now ready to state the results about reverse orbifold construction from \cite{LSn16}.
\begin{theorem}[{\cite[Theorem 5.2]{LSn16}}]\label{reverse}
Let $\g$ be a Lie algebra and $\mathfrak p$ a subalgebra of $\g$. Let $n \in \Z>0$ and let $U$ be a strongly regular holomorphic VOA of central charge $c$. Assume that for any strongly regular holomorphic VOA $W$ of central charge $c$ whose weight one Lie algebra is $\g$, there exists an order $n$ automorphism $\psi$ of $W$ such that the following conditions hold:\\
(a) $\g^\psi\cong \mathfrak p$;  \\
(b) For $1 \leq i \leq n-1$, the lowest $L(0)$-weight of $W^T (\psi^i)$ belongs to $(1/n)\Z>0$;\\
(c) $\tilde W(\psi)$ is isomorphic to $U$.\\
In addition, we assume that any automorphism $\phi\in \Aut (U)$ of order $n$ satisfying the conditions (A), (B) and (C) below belongs to a unique conjugacy class in $\Aut (U )$:\\
(A)  $(U^\phi)_1$ is isomorphic to $\mathfrak p$; \\
(B) For $1 \leq i \leq n-1$, the lowest $L(0)$-weight of $U ^T(\phi^i)$ belongs to $(1/n)\Z>0$;\\ 
(C) $(\tilde U(\phi))_1$ is isomorphic to $\g$.\\
Then any strongly regular holomorphic VOA of central charge $c$ with weight one Lie algebra $\g$ is isomorphic to $\tilde U(\phi)$. In particular, such a holomorphic VOA is unique up to isomorphism.
\end{theorem}

\subsection{Reverse orbifold constructions by inner automorphisms.}
\label{sec:7}
We now begin to prove the uniqueness of the holomorphic VOA of central charge 24 with weight one Lie algebra $F_{4,6}A_{2,2}$. To apply Theorem \ref{reverse}, we let $\g=F_{4}A_{2}$, $\mathfrak p=B_4A_2$ and $W, U$ be  holomorphic VOAs of central charge 24 with weight one Lie algebras $F_{4,6}A_{2,2}$, $A_{8,3}A_{2,1}^2$, respectively. We then need to choose an appropriate automorphism $\psi$ of $W$. Take $h=(\Lambda_4,0)\in F_{4,6}A_{2,2}=W_1$, and define $\sigma_h=\exp(2\pi\sqrt -1h_0)$, where $\Lambda_i$ denotes the fundamental weight of $F_4$. Then we know that $\sigma_h$ is an inner automorphism of $W$. We shall show that $\sigma_h$ is the desired automorphism. To verify that $\sigma_h$ satisfies the condition (b), we need the following result which was proved in \cite{LS}.
\begin{proposition}\label{positive}
 Let $V$ be a strongly regular holomorphic VOA. Assume that the Lie algebra
$\g = V_1$ is semisimple. Let $\g =  \oplus_{i=1}^t \g_i$ be the decomposition into the direct sum of $t$
simple ideals $\g_i$ . Let $\langle V_1\rangle$ be the subVOA of $V$ generated by $V_1$. Let $\tilde h$ be an element in a (fixed) Cartan subalgebra $H$ of $\g$ such that $Spec~ \tilde h_{0}\subseteq (1/T )\Z$ on $V$ for some $T \in \Z>0$. Let $\tilde h_{(i)}$ be the image of $\tilde h$ under the canonical projection from $H$ to $H\cap\g_i$ . We further assume that\\
(1) the conformal vectors of $V$ and $\langle V_1\rangle$ are the same, i.e., $\langle V_1\rangle$ is a full subVOA of $V$;\\
(2) $(\tilde h|\alpha) \geq -1$ for all roots $\alpha\in H$ of $\g$, where $(\cdot|\cdot)$ is the normalized Killing form on $\g$
so that $(\beta|\beta) = 2$ for any long root $\beta$;\\
(3) for some $i$, $-\tilde h_{(i)}$ is not a fundamental weight.\\
Then the lowest $L(0)$-weight of $V^{(\tilde h)}$ is positive, where $V^{(\tilde h)}$ denotes the unique $\sigma_{\tilde h}$-twisted $V$-module.
\end{proposition}
As a consequence, we have the following.
\begin{lemma}\label{conditions}
Let $\sigma_h$ be the automorphism of $W$ defined above. Then $\sigma_h$ is an order two automorphism and satisfies the conditions (a), (b).
\end{lemma}
\pf First, by Theorem 8.6 and Proposition 8.6 in \cite{K}, we know that $\sigma_h|_{F_4A_2}$ is an order two automorphism of  $F_4A_2$ and that $(F_4A_2)^{\sigma_h}\cong B_4A_2$. 
 Moreover, since for any fundamental weight $\Lambda_i$ of $F_4$ we have $(\Lambda_4|\Lambda_i)\in (1/2)\Z$, we then know that $\sigma_h$ is an order two automorphism of $W$.

 We next verify the condition (b). By Proposition 4.1 of \cite{DM5}, we know that the vertex operator subalgebra of W generated by $W_1$ has the same conformal vector as that of $W$. Moreover, it is straightforward to verify that $(\Lambda_4|\alpha)\geq -1$ for any root of $F_4$. Thus, by Proposition \ref{positive}, $\sigma_h$ is an automorphism of $W$ satisfying the condition (b).
 \qed

To verify the condition (c), we need the  following results which were proved in \cite{LS}.
\begin{theorem}\label{dimension}
Let $V$ be a strongly regular holomorphic VOA with central charge $24$, $\tilde h$ be a semisimple element in a Cartan subalgebra $V_1$ such that: (i) $Spec~\tilde h_0\subseteq (1/2)\Z$ and $Spec~\tilde h_0\nsubseteq \Z$; (ii) $\langle \tilde h, \tilde h \rangle\in \Z$, where $\langle\cdot, \cdot\rangle$ denotes the unique symmetric invariant bilinear form on $V$ such that $\langle\1, \1\rangle=-1$; (iii) The lowest weight of $V^{(\tilde h)}$ is positive; (iv) $(V, \sigma_{\tilde h})$ satisfies the orbifold condition. Then we have $\dim V_1+\dim \tilde V(\sigma_{\tilde h})_1=3\dim V_1^{\sigma_{\tilde h}}+24(1-\dim (V^{(\tilde h)})_{1/2})$.

\end{theorem}

As a consequence, we have the following.
\begin{lemma}\label{conditions1}
Let $W$ and $h\in W_1$ be as before. Then the automorphism $\sigma_h$ satisfies  the condition (c).
\end{lemma}
\pf By Proposition 5.3 of \cite{LS}, we know that $(W,\sigma_h)$ satisfies the orbifold condition. Thus, $h$ satisfies the conditions in Theorem \ref{dimension}, and we have $\dim \tilde W(\sigma_h)_1=3\dim W_1^{\sigma_h}+24(1-\dim (W^{(h)})_{1/2})-\dim W_1$. We next determine the dimension of $(W^{(h)})_{1/2}$. By Lemma 2.7 of \cite{LS}, the lowest conformal weight of $\left(L_{F_4}(6, \lambda_1)\otimes L_{sl_3}(2,\lambda_2)\right)^{(h)}$ is equal to $
l({\lambda_1, \lambda_2})+\sum_{i=1}^2 \min\{(h_{(i)}|\mu)\,|\,\mu\in \Pi(\lambda_i)\}+\langle h|h\rangle/2,
$
where $l({\lambda_1, \lambda_2})$ is the lowest conformal weight of $L_{F_4}(6,\lambda_1)\otimes L_{sl_3}(2, \lambda_2)$ and $\Pi(\lambda_i)$ is the set of all weights of the irreducible
module $L(\lambda_i)$ with the highest weight $\lambda_i$. By a direct computation, the possible pairs $(\lambda_1, \lambda_2)$ such that $L_{F_4}(6,\lambda_1)\otimes L_{sl_3}(2, \lambda_2)$ has integral conformal weights are $(0,0)$, $(\Lambda_4, \dot{\Lambda}_1+\dot{\Lambda}_2)$, $(3\Lambda_4, \dot{\Lambda}_1+\dot{\Lambda}_2)$, $(4\Lambda_4, 0)$, $(6\Lambda_4,\dot{\Lambda}_1+\dot{\Lambda}_2)$, $(2\Lambda_3+\Lambda_4, \dot{\Lambda}_1+\dot{\Lambda}_2)$, $(3\Lambda_3, 0)$, $(\Lambda_2+\Lambda_4, \dot{\Lambda}_1)$, $(\Lambda_2+\Lambda_4, \dot{\Lambda}_2)$, $(\Lambda_2+2\Lambda_4, 2\dot{\Lambda}_2)$, $(\Lambda_2+2\Lambda_4, 2\dot{\Lambda}_1)$, $(\Lambda_1+3\Lambda_4, 0)$, $(\Lambda_1+\Lambda_3+2\Lambda_4, \dot{\Lambda}_1)$, $(\Lambda_1+\Lambda_3+2\Lambda_4, \dot{\Lambda}_2)$, $(\Lambda_1+\Lambda_2, 0)$, $(2\Lambda_1, 2\dot{\Lambda}_2)$, $(2\Lambda_1, 2\dot{\Lambda}_1)$, $(2\Lambda_1+\Lambda_3, \dot{\Lambda}_1+\dot{\Lambda}_2)$, the lowest conformal weights are $0$, $1$, $2$, $2$, $4$, $3$, $3$, $2$, $2$, $3$, $3$, $3$, $3$, $3$, $2$, $2$, $2$, $3$, respectively, where $\dot{\Lambda}_i$ denotes the fundamental weights of $sl_3$. Since $(\Lambda_4|\Lambda_1)=1$, $(\Lambda_4|\Lambda_2)=2$, $(\Lambda_4|\Lambda_3)=3/2$ and  $(\Lambda_4|\Lambda_4)=1$,  we have $\dim (W^{(h)})_{1/2}=0$ by Lemma 4.1 in \cite{KLL}. It follows that $\tilde W(\sigma_h)_1$ is a semisimple Lie algebra of dimension $96$.

We next determine the Lie algebra structure of $\tilde W(\sigma_h)_1$. Assume that $\tilde W(\sigma_h)_1\cong \g_{1,k_1}\oplus \cdots \oplus \g_{t, k_t}$, where $\g_{i, k_i}$ means the level of $\g_i$ is equal to $k_i$. By Theorem 3 in \cite{DM5}, we have $h^{\vee}_i/k_i=3$, where $h^{\vee}_i$ denotes the dual Coxeter number of $\g_i$. It follows that the possible Lie subalgebras of $W(\sigma_h)_1$ are $A_2$, $B_2$, $A_5$, $C_5$, $D_4$, $A_8$, $B_5$, $F_4$, $E_6$. Moreover,  we know that there exists an automorphism $a_{W, \sigma_h}$ of $W(\sigma_h)_1$ such that the fixed point subalgebra $W(\sigma_h)_1^{a_{W, \sigma_h}}$ is isomorphic to $B_4A_2$ (see Subsection \ref{sec1}). By Propositions 3.1, 3.3 of \cite{KLL}, $W(\sigma_h)_1$ should be isomorphic to $A_{8,3}A_{2,1}A_{2,1}$. The proof is complete.
\qed

\subsection{Conjugacy classes of the automorphism group $\Aut (U)$}
In this subsection, we shall prove that automorphisms of $U$ satisfying the conditions (A), (B) and (C) belong to a unique conjugacy class in $\Aut (U )$. Let $\Phi$ be an order two automorphism of $U$ such  that $U_1^\Phi$ is isomorphic to $B_4A_2$. Then we know that  $\Phi|_{U_1}$ is conjugate to $\theta\otimes \sigma$ under $\Aut (U_1)$ \cite[Chapter X, Theorem 6.1]{He}. In the following, we will further  prove that $\Phi$ is conjugate to one of liftings of ${\theta\otimes \sigma}$ under $\Aut(U)$. First, by the similar argument in Lemma 4.2.8 of \cite{Y}, we have:
\begin{theorem}\label{stable}
Let $V$ be a VOA, $\tilde V$  an extension of $V$ and $\psi$ an automorphism of $V$. Assume that $\tilde V$ viewed as a $V$-module has the decomposition $\tilde V=V\oplus M^1\oplus \cdots\oplus M^k$, where $M^0=V, M^1,\cdots, M^k$ are nonisomorphic irreducible $V$-modules, and that there exists an automorphism $\tilde \psi$ of $\tilde V$ such that $\tilde \psi(V)=V$ and $\tilde \psi|_V=\psi$. Then
$$\{M^0, M^1,\cdots, M^k\}=\{M^0\circ \psi, M^1\circ \psi,\cdots, M^k\circ \psi\}.$$
\end{theorem}

As a consequence, we have the following.
\begin{theorem}\label{Thm7.7}
Let $\Phi$ be as above. Then $\Phi$  is conjugate to one of liftings of ${\theta\otimes \sigma}$ under $\Aut(U)$.
\end{theorem}
\pf Since $\Phi|_{U_1}$ is conjugate to $\theta\otimes \sigma$ under $\Aut (U_1)$, there exists an automorphism $f$ of $U_1$ such that $\Phi|_{U_1}=f(\theta\otimes \sigma) f^{-1}$.  It is well-known that $f$ has the form $\exp(2\pi \sqrt{-1}u)\mu$ or $\exp(2\pi \sqrt{-1}u)\mu\circ (1\otimes \sigma)$, where $u$ is an element of a Cartan subalgebra of $U_1$ and $\mu$ is one of the following automorphisms \begin{align*}
&\theta\otimes 1\otimes 1, \theta\otimes \varphi\otimes 1, \theta\otimes 1\otimes \varphi, \theta\otimes \varphi\otimes \varphi, 1\otimes\varphi \otimes 1,1\otimes 1 \otimes \varphi, 1\otimes \varphi \otimes \varphi, 1\otimes 1 \otimes 1.
\end{align*}
Here, $\varphi$ denotes the diagram automorphism of $sl_3$. Note that $(1\otimes \sigma)(\theta\otimes \sigma)(1\otimes \sigma)^{-1}=\theta\otimes \sigma$, we then may assume that $f$ has the form $\exp(2\pi \sqrt{-1}u)\mu$.

By assumption, $U$ viewed as an $L_{sl_9}(3,0)\otimes L_{sl_3}(1,0)\otimes L_{sl_3}(1,0)$-module has the following decomposition
{\tiny\begin{align*}
&\widetilde{L_{sl_9}(3, 0)}\ot L_{sl_3}(1,0)\ot L_{sl_3}(1,0)\oplus \widetilde{L_{sl_9}(3, 3\lddot_1)}\ot L_{sl_3}(1,\ldot_1)\ot L_{sl_3}(1,\ldot_2)\oplus \widetilde{L_{sl_9}(3, 3\lddot_2)}\ot L_{sl_3}(1,\ldot_2)\ot L_{sl_3}(1,\ldot_1)\\
&\oplus \tau_1\ot L_{sl_3}(1,\ldot_1)\ot L_{sl_3}(1,\ldot_1)\oplus \tau_2\ot L_{sl_3}(1,\ldot_2)\ot L_{sl_3}(1,\ldot_2)\oplus (\widetilde{L_{sl_9}(3, 3\lddot_1)}\times \tau_1)\ot L_{sl_3}(1,\ldot_2)\ot L_{sl_3}(1,0)\\
&\oplus (\widetilde{L_{sl_9}(3, 3\lddot_2)}\times \tau_1)\ot L_{sl_3}(1,0)\ot L_{sl_3}(1,\ldot_2)\oplus (\widetilde{L_{sl_9}(3, 3\lddot_1)}\times \tau_2)\ot L_{sl_3}(1,0)\ot L_{sl_3}(1,\ldot_1)\\
&\oplus (\widetilde{L_{sl_9}(3, 3\lddot_2)}\times \tau_2)\ot L_{sl_3}(1,\ldot_1)\ot L_{sl_3}(1,0).
\end{align*}}
Note that the inner automorphism $\exp(2\pi \sqrt{-1}u)$ of $U_1$ can be lifted to an automorphism of $U$ and that  {\tiny$$\widetilde{L_{sl_9}(3, 3\lddot_1)}\ot L_{sl_3}(1,\ldot_1)\ot L_{sl_3}(1,\ldot_2)\circ (\theta\otimes \varphi\otimes 1)(\theta\otimes \sigma)(\theta\otimes \varphi\otimes 1)^{-1} \cong \widetilde{L_{sl_9}(3, 3\lddot_2)}\ot L_{sl_3}(1,\ldot_1)\ot L_{sl_3}(1,\ldot_2).$$}It follows from Theorem \ref{stable} that  $\mu$ cannot be equal to $\theta\otimes \varphi\otimes 1$. Note also that
{\tiny
\begin{align*}
&\widetilde{L_{sl_9}(3, 3\lddot_1)}\ot L_{sl_3}(1,\ldot_1)\ot L_{sl_3}(1,\ldot_2)\circ (1\otimes \varphi\otimes 1)(\theta\otimes \sigma)(1\otimes \varphi\otimes 1)^{-1} \cong \widetilde{L_{sl_9}(3, 3\lddot_2)}\ot L_{sl_3}(1,\ldot_1)\ot L_{sl_3}(1,\ldot_2),\\
&\widetilde{L_{sl_9}(3, 3\lddot_1)}\ot L_{sl_3}(1,\ldot_1)\ot L_{sl_3}(1,\ldot_2)\circ (\theta\otimes 1\otimes \varphi)(\theta\otimes \sigma)(\theta\otimes 1\otimes \varphi)^{-1} \cong \widetilde{L_{sl_9}(3, 3\lddot_2)}\ot L_{sl_3}(1,\ldot_1)\ot L_{sl_3}(1,\ldot_2),\\
&\widetilde{L_{sl_9}(3, 3\lddot_1)}\ot L_{sl_3}(1,\ldot_1)\ot L_{sl_3}(1,\ldot_2)\circ (1\otimes 1\otimes \varphi)(\theta\otimes \sigma)(1\otimes 1\otimes \varphi)^{-1} \cong \widetilde{L_{sl_9}(3, 3\lddot_2)}\ot L_{sl_3}(1,\ldot_1)\ot L_{sl_3}(1,\ldot_2).
\end{align*}
}
It follows that $\mu$ cannot be equal to $1\otimes \varphi\otimes 1$, $\theta\otimes 1\otimes \varphi$ or $1\otimes 1\otimes \varphi$. As a result, $\mu$ is equal to $\theta\otimes 1\otimes 1$, $\theta\otimes \varphi\otimes \varphi$, $1\otimes 1\otimes 1$ or $1\otimes \varphi\otimes \varphi$. Note that in these cases we have $\mu(\theta\otimes \sigma)\mu^{-1}=\theta\otimes \sigma$. Thus, we have $\Phi|_{U_1}=\exp(2\pi \sqrt{-1}u)(\theta\otimes \sigma) \exp(2\pi \sqrt{-1}u)^{-1}$. It follows that $\exp(2\pi \sqrt{-1}u)^{-1}\Phi\exp(2\pi \sqrt{-1}u)$ is a lifting of $\theta\otimes \sigma$. The proof is complete.
\qed
\vskip.25cm
We next  prove that any liftings $g_1$, $g_2$ of $\theta\otimes \sigma$ of order two are conjugate under $\Aut (U)$.
\begin{theorem}\label{conjugacy}
Let $g_1$, $g_2$ be  liftings of $\theta\otimes \sigma$ of order two. Then  $g_1$, $g_2$ are conjugate under $\Aut (U)$.
\end{theorem}
\pf By Theorem \ref{identity}, we know that the order of $g_1g_2^{-1}$ is equal to $1$ or $3$. If the order  of $g_1g_2^{-1}$ is equal to $1$, then $g_1=g_2$. If the order of $g_1g_2^{-1}$ is equal to $3$, then we have $g_2g_1g_2g_1g_2=g_1$. This implies that $(g_1g_2)^{-1}g_2(g_1g_2)=g_1$. Therefore, we always have $g_1$, $g_2$ are conjugate under $\Aut (U)$. \qed

We now in a position to prove our main result in this section.
\begin{theorem}\label{uniqueF4}
Let $W^1, W^2$ be  holomorphic VOAs  of central charge $24$ with weight one Lie algebras $F_{4,6}A_{2,2}$. Then $W^1, W^2$ are isomorphic.
\end{theorem}
\pf This follows immediately from Theorems \ref{reverse}, \ref{conjugacy} and Lemmas \ref{conditions}, \ref{conditions1}.
\qed

\vskip.5cm
As another application of Theorem \ref{identity}, we can also establish the uniqueness of the holomorphic VOA of central charge $24$ with one Lie algebra $E_{7,3}A_{5,1}$.
\begin{theorem}
Let $W^1, W^2$ be  holomorphic VOAs of central charge $24$ with weight one Lie algebras $E_{7,3}A_{5,1}$. Then  $W^1, W^2$ are isomorphic.
\end{theorem}
\pf To apply the reverse orbifold construction method, we take $W$ and $U$ to be holomorphic VOAs of central charge $24$ with weight one Lie algebras $E_{7,3}A_{5,1}, A_{8,3}A_{2,1}^2$, respectively, and take $\g$, $\mathfrak{p}$ to be Lie algebras $E_7A_5$, $A_7A_2^2U(1)$, respectively. Let $h=\frac{1}{2}(\dddot \Lambda_2, \ddddot\Lambda_3)\in W_1$ and take $\psi$ to be the inner automorphism $\sigma_h$ of $W$, where $\dddot \Lambda_i, \ddddot \Lambda_i$ denote the fundamental weights of $E_7$ and $A_5$, respectively. It was proved in \cite{LS} that $\sigma_h$ is an involution of $W$ and satisfies  the conditions (a), (b), and (c).

We next prove that any involutions $\Phi_1, \Phi_2$ of $U$ satisfying the conditions (A), (B), (C) are conjugate under $\Aut (U)$. By assumption, we have $U^{\Phi_1}$ is isomorphic to $A_7A_2^2U(1)$. Note that the fixed point subalgebra of $A_{8,3}A_{2,1}^2$ under the action of the inner automorphism of $\exp(\pi \sqrt{-1} (\ddot\Lambda_1+\ddot\Lambda_2))$ is also $A_7A_2^2U(1)$, where $\ddot\Lambda_i$ denotes the fundamental weight of $A_8$. Since $\Phi_1$ has order $2$, it follows that $\Phi_1|_{U_1}$ is conjugate to $\exp(\pi \sqrt{-1} (\ddot\Lambda_1+\ddot\Lambda_2))$ under $\Aut (U_1) $ \cite{He}. In particular, there exists an automorphism $f$ of $U_1$ such that $\Phi_1|_{U_1}=f\exp(\pi \sqrt{-1} (\ddot\Lambda_1+\ddot\Lambda_2)) f^{-1}$.  It is well-known that $f$ has the form $\exp(2\pi \sqrt{-1}u)\mu$ or $\exp(2\pi \sqrt{-1}u)\mu\circ (1\otimes \sigma)$, where $u$ is an element of a Cartan subalgebra of $U_1$ and $\mu$ is one of the following automorphisms \begin{align*}
&\theta\otimes 1\otimes 1, \theta\otimes \varphi\otimes 1, \theta\otimes 1\otimes \varphi, \theta\otimes \varphi\otimes \varphi, 1\otimes\varphi \otimes 1,1\otimes 1 \otimes \varphi, 1\otimes \varphi \otimes \varphi, 1\otimes 1 \otimes 1.
\end{align*}
Here, $\varphi$ denotes the diagram automorphism of $sl_3$. By the same arguments as in Theorem \ref{Thm7.7}, we can take $f$ to be $\exp(2\pi \sqrt{-1}u)$ or $\exp(2\pi \sqrt{-1}u)\theta\otimes \sigma$. Since $\exp(2\pi \sqrt{-1}u)$ and $\exp(2\pi \sqrt{-1}u)\theta\otimes \sigma$ can be lifted to automorphisms of $U$, it follows that $\Phi_1$ is conjugate to one of liftings of $\exp(\pi \sqrt{-1} (\ddot\Lambda_1+\ddot\Lambda_2))$. Similarly, $\Phi_2$ is conjugate to one of liftings of $\exp(\pi \sqrt{-1} (\ddot\Lambda_1+\ddot\Lambda_2))$. Thus, by the same argument in the proof of Theorem \ref{conjugacy}, we can prove that $\Phi_1, \Phi_2$ are conjugate under $\Aut(U)$. Hence, by Theorem \ref{reverse},  any two holomorphic VOAs of central charge $24$ with  the weight one Lie algebra $E_{7,3}A_{5,1}$ are isomorphic. The proof is complete.
\qed


\end{document}